\documentclass[11pt, a4paper]{article}
\usepackage[english]{babel}
\usepackage[a4paper, left=3cm, right=3cm, top=3cm, bottom=3cm]{geometry}
\usepackage{amsmath, amssymb, amsthm, bbm, multicol, rotating, tabularx, amsfonts,graphics, enumitem}
\parskip\medskipamount
\usepackage[all]{xy}
\usepackage[english]{babel}
\usepackage[pdfpagelabels]{hyperref}

\newtheorem{theorem}{Theorem}[section]
\newtheorem{Theorem}{Main Result}
\newtheorem{prop}[theorem]{Proposition}
\newtheorem{lemma}[theorem]{Lemma}
\newtheorem{cor}[theorem]{Corollary}

\theoremstyle{remark}

\theoremstyle{definition}
\newtheorem{definition}[theorem]{Definition}

\newtheorem*{cplx*}{Theorem~\ref{thm:cplxofgroups}}
\newtheorem*{cones*}{Theorems~\ref{prop:ultraproducts} and \ref{thm:cones}}

\def\<{\langle}
\def\>{\rangle}

\newcommand{\cA}{\mathcal{A}}

\def\qed{{\hfill\hphantom{.}\nobreak\hfill$\Box$}}

\newcommand{\A}{\mathbb{A}} 
\newcommand{\R}{\mathbb{R}}

\newcommand{\N}{\mathbb{N}}

\newcommand{\RS}{\mathrm{R}}
\newcommand{\Q}{\mathbb{Q}}
\newcommand{\MS}{\mathbb{A}} 
\newcommand{\sW}{\overline{W}} 
\newcommand{\aW}{W} 
\newcommand{\Cf}{\mathcal{{C}}_{f}} 
\newcommand{\define}{\mathrel{\mathop:}=}

\newcommand{\uprod}{^*\!} 
\newcommand{\ud}{\uprod d} 
\newcommand{\Cone}{\mathrm{Cone}} 
\newcommand{\Ucone}{\mathrm{UCone}} 

\begin{document}

\author{
Petra N. Schwer (n\'ee Hitzelberger)\thanks{The first author is supported by the ``SFB 478 Geometrische Strukturen in der Mathematik'' at the Institute of Mathematics,  University of M\"unster} 
\and
Koen Struyve\thanks{The second author is supported by the Fund for Scientific Research -- Flanders (FWO -- Vlaanderen)}}
\title{\bf $\Lambda$--buildings and base change functors}

\maketitle

\begin{abstract}
We prove an analog of the base change functor of $\Lambda$--trees in the setting of generalized affine buildings. The proof is mainly based on local and global combinatorics of the associated spherical buildings. As an application we obtain that the class of generalized affine buildings is closed under taking ultracones and asymptotic cones. Other applications involve a complex of groups decompositions and fixed point theorems for certain classes of generalized affine buildings.

\end{abstract}

\section{Introduction}

The so-called $\Lambda$--trees have been studied by Alperin and Bass \cite{AlperinBass}, Morgan and
Shalen \cite{MorganShalen} and others and have proven to be a useful tool in understanding
properties of groups acting nicely on such spaces. $\Lambda$--trees are a natural generalization of $\R$--trees. Here $\Lambda$ is an arbitrary ordered abelian group replacing the copies of the real line in the concept of an $\R$--tree or the geometric realizations of simplicial trees.

Since simplicial trees are precisely the one-dimensional examples of affine
buildings and real trees the one-dimensional $\R$--buildings, 
it was natural to ask whether there is a higher dimensional object
generalizing $\Lambda$--trees and affine buildings at the same time.

These objects, the so-called $\Lambda$--affine buildings or 
\emph{generalized affine buildings}, where introduced by Curtis Bennett \cite{Bennett} and recently studied by the first author in \cite{PetraThesis} and \cite{Convexity2}. A recent application of them is a short proof of the Margulis conjecture by Kramer and Tent \cite{KramerTent}.

In the present paper, we address a generalization of an important geometric property of $\Lambda$--trees: the existence of a base change functor. Easy to prove in the tree case, see for example \cite{Chiswell}, the generalization to $\Lambda$--affine buildings turns out to be much harder.

We will prove that a morphism $e:\Lambda \to\Gamma$ of ordered abelian groups 
gives rise to a base change functor $\phi$ mapping a generalized affine building $X$ 
defined over $\Lambda$ to another building $X'$ which is defined over $\Gamma$.
In case $e$ is an epimorphism, we will see, that the pre-image $X''$ under $\phi$ of a point  in $X'$ is again a generalized affine building, defined over the kernel of $e$.

After having established our main results, we will present several applications.
One of the consequences of our base change theorem is the proof of the fact that the class of generalized affine buildings is closed under taking asymptotic cones and ultracones.


\subsection{Our main results}\label{MR}

A set $X$ together with a collection $\cA$  of charts $f:\MS\to X$ from a model space $\MS$ into $X$ is a \emph{generalized affine building} if certain compatibility and richness conditions, as stated in Definition~\ref{Def_LambdaBuilding}, are satisfied. These conditions imply the existence of a spherical building $\partial_\cA X$ at infinity and of spherical buildings $\Delta_{x}X$, called  \emph{residues}, around each point $x\in X$, see Section~\ref{section:lg}. The model space $\MS$ is defined with respect to a spherical root system $\RS$ and a totally ordered abelian group $\Lambda$. Therefore it is sometimes denoted by $\MS(\RS,\Lambda)$. As a set it is isomorphic to the space of formal sums $
\left\{\sum_{\alpha\in B} \lambda_\alpha\alpha : \lambda_\alpha\in\Lambda\right\},
$
 where $B$ is a basis of $\RS$. 

The model space carries an action of an affine Weyl group $\aW$ and the transition maps of charts are given by elements of $\aW$. 
One of the conditions $(X,\cA)$ has to satisfy is, that every pair of points $x$ and $y$ is contained in a common apartment $f(\MS)$ with $f\in\cA$. 

Given a morphism $e:\Lambda\to\Lambda'$ of ordered abelian $\Q [\{\alpha^\vee(\beta)\}_{\alpha,\beta\in\RS}]$--modules $\Lambda$ and $\Lambda'$
and let $(X,\cA)$ be an affine building with model space $\MS(\RS,\Lambda)$ (or shortly $\MS$) and distance function $d$ which is induced by the standard distance on the model space. Then we have the following:

\begin{theorem}\label{thm:combi}
There exists a $\Lambda'$--building $(X',\cA')$ and (functorial) map $\phi:X \rightarrow X'$ such that $\phi$ maps $\cA$ to $\cA'$ and such that 
$$
d'(\phi(x),\phi(y)) = e(d(x,y)) 
$$ 
for all $x,y \in X$, where $d$, $d'$ are the distance functions defined respectively on $(X,\cA)$ and $(X',\cA')$. Furthermore, the spherical buildings $\partial_\cA X$ and $\partial_{\cA'} X'$ at infinity are isomorphic.

\end{theorem}


The map $\phi$ will be referred to as the \emph{base change functor} associated to $e$.

Every morphism of abelian groups can be written as the composition of an epimorphism followed by a monomorphism. The kernel of an epimorphism of ordered abelian groups is convex in the sense that given some $x>0$ in this kernel then $x\geq y\geq 0$ implies that $y$ is also contained in the kernel. Therefore the ordering of an abelian group $\Lambda$ induces an order on the quotient of $\Lambda$ by the kernel of an epimorphism. Hence morphisms of ordered abelian groups can also be decomposed into an epimorphism followed by a monomorphism. 

We will prove the two cases, of an epimorphism and a monomorphism, separately in the first and third Main Result. Theorem~\ref{thm:combi} is a direct consequence of the combination of both.


In case $e:\Lambda\to\Lambda'$ is an epimorphism one defines the base change functor $\phi$ and the building $X'$ is as follows. Let two points $x,y\in X$ be equivalent, denoted by $x\sim y$, when $d(x,y)\in\ker(e)$ and let $X'$ be the quotient of $X$ defined by this equivalence relation. Defining a metric $d'$ on $X'$ by $d'(\phi(x),\phi(y))\define e(d(x,y))$ the quotient map $\phi:X\to X'$ turns out to satisfy the properties needed for the first Main Result. 

It is possible to define a set of charts $\cA'$ from $\MS'=\MS(\RS,\Lambda')$ into $X'$ such that the following theorem holds. For details see Section~\ref{section:firstMain}.

\begin{Theorem}\label{theorem:firstmain}
Let $e: \Lambda \rightarrow \Lambda'$ be an epimorphism of ordered abelian groups and $(X,\cA)$ a $\Lambda$--building. Then the following hold:
\begin{enumerate}
\item\label{item:mr11}
There exists a $\Lambda'$--building $(X',\cA')$ and a map $\phi:X \twoheadrightarrow X'$, called the \emph{base change functor} associated to $e$,  such that the apartment system $\cA$ is mapped onto $\cA'$ by $\phi$ and such that $\phi$ satisfies
\begin{equation*}\label{equ:mr1}
d'(\phi(x),\phi(y)) = e(d(x,y)) 
\end{equation*}
for all $x,y \in X$, where $d$, $d'$ are the distance functions defined respectively on $(X,\cA)$ and $(X',\cA')$.  Moreover, the spherical buildings $\partial_\cA X$ and $\partial_{\cA'} X'$ at infinity are isomorphic.


\item\label{item:mr13}
These base change functors act as a functor on the category of $\Lambda$--buildings to the one of the $\Lambda'$--buildings. In particular, let $G$ be a group acting on $X$ by isometries, then $G$ also acts on $X'$ by isometries and the map $\phi$ is $G$--equivariant. 
\end{enumerate}
\end{Theorem}

As mentioned above, the kernel of an ordered abelian group is again such. It turns out that one can prove that the fibers of the base change functor are again generalized affine buildings. Details of the proof can be found in Section~\ref{section:secondMain}.

\begin{Theorem}\label{theorem:secondmain}
Let $\phi$ be a base change functor associated to an ordered abelian group epimorphism $e: \Lambda \rightarrow \Lambda'$, applied to a $\Lambda$--building $(X,\cA)$. 
For all elements $x$ of $X$ the following is true.
\begin{enumerate}
 \item\label{item:mr21} The set $X'' =\phi^{-1}(\phi(x))$ admits a set of charts $\cA''$ making it into a $\ker(e)$--building with as distance function $d''$ the distance function inherited from $(X,\cA)$.
 \item\label{item:mr22} There is a natural isomorphism between $\partial_{\cA''}X''$ and $\Delta_{\phi(x)}X'$, where $(X',\cA')$ is as in the first Main Result.
\end{enumerate}
\end{Theorem}

One can prove a result similar to the first main one for monomorphisms of ordered abelian groups. The construction of the building $X'$ is again of explicit nature. 
The basic idea is to take the product of the old charts with the new, enlarged, model space and consider equivalence classes of these products.
Since this construction is more involved than the one in the first Main Result, let us postpone details to Section~\ref{section:thirdMain}. There we will prove

\begin{Theorem}\label{theorem:thirdmain}
Let $e: \Lambda \rightarrow \Lambda'$ be a monomorphism of ordered abelian groups and $(X,\cA)$ a $\Lambda$--building. Then the following assertions hold
\begin{enumerate}
\item\label{item:mr31} 
There exists a $\Lambda'$--building $(X',\cA')$ and a map $\phi:X \rightarrow X'$ satisfying 
\begin{equation*}
d'(\phi(x),\phi(y)) = e(d(x,y)) 
\end{equation*}
for all $x,y \in X$, where $d$, $d'$ are the distance functions defined respectively on $(X,\cA)$ and $(X',\cA')$. 
Further $\phi$ maps the apartment system $\cA$ to $\cA'$ and the spherical buildings $\partial_\cA X$ and $\partial_{\cA'} X'$ at infinity are isomorphic.

\item\label{item:mr32} 
These base change functors act as a functor on the category of $\Lambda$--buildings to the one of the $\Lambda'$--buildings, but only for isometries mapping apartments to apartments. In particular, let $G$ be a group acting on $X$ by isometries stabilizing the system of apartments, then $G$ acts on $X'$ by isometries stabilizing the system of apartments and the map $\phi$ is $G$--equivariant.
\end{enumerate}
\end{Theorem}

The proofs of the Main Results 1 to 3 can be found in Sections~\ref{section:firstMain} to \ref{section:thirdMain}.
Before defining generalized affine buildings in Section~\ref{Sec:Prelim}, we will use the penultimate subsection of the introduction to state our main applications. The last subsection will be devoted to an example clarifying the main results.

\subsection{Applications}\label{section:pres_apps}

Let us quickly summarize the applications proved in Section~\ref{sec:applications}.

\subsubsection*{Asymptotic cones}

Asymptotic cones of metric spaces capture the `large scale structure' of the underlying space.  The main idea goes back to the notion of convergence of metric spaces by Gromov in the early 80's (see \cite{Gromov2}) and was later generalized using ultrafilters by van den Dries and Wilkie \cite{vandenDriesWilkie}.
Asymptotic cones provide interesting examples of metric spaces and have proven useful in the context of geometric group theory. 

In Section~\ref{sec:cones} we will prove, using the base change functor, that the class of generalized affine buildings is closed under ultraproducts, asymptotic cones and ultracones. The main results read as follows:

\begin{cones*}

\begin{itemize}
 \item The ultraproduct of a sequence $(X_i,\cA_i)_{i\in I}$ of $\Lambda_i$--affine buildings defined over the same root system $\RS$ is again a generalized affine building over $\RS$. 
 \item Asymptotic cones and ultracones of generalized affine buildings are again such.
 \item Furthermore, if $(X,\cA)$ is modeled on $\MS(\RS,\Lambda)$, then its asymptotic cone $\Cone(X)$ is modeled an $\MS(\RS,\Cone(\Lambda))$ and its ultracone $\Ucone(X)$ on $\MS(\RS,\Ucone(\Lambda))$.
\end{itemize}
\end{cones*}

In particular we have shown that asymptotic cones of $\R$--buildings are again such, and with this yielding an alternative proof of the same result shown earlier with completely different methods by Kleiner and Leeb (see~\cite{KleinerLeeb}).

\subsubsection*{Fixed point theorems}
The base change functors can be used to reduce problems of generalized affine buildings to the (easier) case of $\R$--buildings. In Section~\ref{section:reducing} we illustrate this with a fixed point theorem for a certain class of $\Lambda$--buildings (we postpone the description of this class to the aforementioned section). The result then reads:

\begin{theorem}
A finite group of isometries of a generalized affine building $(X,\cA)$ of this class admits a fixed point.
\end{theorem}

\subsubsection*{Complex of groups decompositions}
One can use the first Main Result to conclude that groups acting nicely
on certain affine buildings do admit a complex of groups decomposition.
We will not carry out the details of the proof, but let us make the
statement a bit more precise.

Assume that $(X,\cA)$ is modeled over an abelian group 
$\Lambda\define\R\times\Lambda'$, where the two components are
ordered  lexicographically,
and assume further that the image of the base change functor associated
to the projection $e:\Lambda \to \R$ is a simplicial affine building.

Then, if $G$ is a subgroup of the automorphism group of $X$ such that
the induced action on $(X',\cA')$ is simplicial, the group $G$ has a
complex of groups decomposition where each vertex group acts on a
$\Lambda'$--building. In addition, if the action of $G$ on $X$ is free,
then each vertex group acts freely on a $\Lambda'$--building.

\subsection{An example}
Let us illustrate the main results with an example in an algebraic setting. We start with describing a class of generalized affine buildings (following~\cite[p. 97]{BennettThesis}). Let $K$ be a field with a valuation to an ordered abelian group $\Lambda$. Then one can define root group data with valuation for the special linear group $\mathsf{SL}(n,K)$ ($n \leq 2$). These data give rise to an $(n-1)$-dimensional $\Lambda$--building admitting a natural action of $\mathsf{SL}(n,K)$. The spherical building at infinity here is the spherical Tits building associated to $\mathsf{SL}(n,K)$. The thick residues are isomorphic to the spherical Tits building associated to $\mathsf{SL}(n,K_\nu)$, where $K_\nu$ is the residue field of the pair $(K,\nu)$.

Let $\Lambda_1$ be the lexicographically ordered group $\mathbb{R} \times \mathbb{R}$. Let $K$ be some field with a $\Lambda_1$--valued valuation $\nu$. An example could be a rational function field $k(t^{\Lambda_1})$ in the variable $t$ allowing powers in $\Lambda_1$. As mentioned above one can associate a $\Lambda_1$--building $(X,\cA)$ to the group $\mathsf{SL}(n,K)$.

Let $\Lambda_2$ be the group $\mathbb{R} \times \mathbb{R} \times \mathbb{R}$ ordered lexicographically. Let $e$ be the map $$e:\Lambda_1 \to \Lambda_2: (a,b) \mapsto (0,a,0).$$ This is a morphism of ordered abelian groups which can be split up in an epimorphism $e_1$ and monomorphism $e_2$ (so $e= e_2 \circ e_1$) where \begin{align*} e_1&:\Lambda_1 \to \mathbb{R}: (a,b) \mapsto a, \\ e_2&: \mathbb{R} \to \Lambda_2 : a \mapsto (0,a,0). \end{align*}

The image of $(X,\cA)$ under the base change functor for $e_1$ is the generalized affine building for $\mathsf{SL}(n,K)$, but now using the (real) valuation $e_1 \circ \nu$. Similarly the image under the base change functor for $e$ will be the generalized affine building associated to $\mathsf{SL}(n,K)$ with the valuation $e \circ \nu$ with values in $\Lambda_2$. Theorem~\ref{thm:combi} mentions that the spherical buildings at infinity of these generalized affine buildings are isomorphic, this is reflected in this example by the special linear group staying the same. 

To illustrate the second Main Result consider a point $x$ with a thick residue in the generalized affine building associated to $\mathsf{SL}(n,K)$ with the valuation $e_1 \circ \nu$. This residue is isomorphic to the spherical Tits building for $\mathsf{SL}(n,K_{e_1 \circ \nu})$. The second Main Result now states that the preimage of this point is an $\R$--building with as building at infinity this residue. This $\R$--building is the one defined by the (real) valuation on the residue field $K_{e_1 \circ \nu}$ induced by the  valuation $\nu$.

\section{Preliminaries}\label{Sec:Prelim}
In this section we will define $\Lambda$--buildings and state some basic results about them for use in later sections. For a detailed study of generalized affine buildings and proofs of the results in this introductory section we refer to \cite{Bennett} and \cite{PetraThesis}.

\subsection{Definition of apartments and buildings}\label{section:defs}

We will first define the model space for apartments in $\Lambda$--buildings and examine its metric structure. We conclude this subsection with the definition of a $\Lambda$--building. 

For a (not necessarily crystallographic) spherical root system $\RS$ let $F$ be a subfield of the reals containing the set $\{\langle \beta, \alpha^\vee\rangle : \alpha, \beta \in \RS\}$ of all evaluations of co-roots on roots. Notice that $F$ can always be chosen to be the quotient field of $\Q[\{\langle \beta, \alpha^\vee\rangle : \alpha, \beta \in \RS\}]$. If $\RS$ is crystallographic this is $F=\Q$.  Assume that $\Lambda$ is a (non-trivial) totally ordered abelian group admitting an $F$--module structure and define the \emph{model space} of a generalized affine building of type $\RS$ to be the set
$$
\MS(\RS,\Lambda) = \mathrm{span}_F(\RS)\otimes_F \Lambda.
$$ 

We will often abbreviate $\MS(\RS,\Lambda)$ by $\MS$. A fixed basis $B$ of the root system $\RS$ provides natural coordinates for the model space $\MS$. The vector space of formal sums 
$$
\left\{\sum_{\alpha\in B} \lambda_\alpha\alpha : \lambda_\alpha\in\Lambda \right\}
$$ 
is canonically isomorphic to $\MS$.
The evaluation of co-roots on roots $\langle\cdot, \cdot\rangle$ is linearly extended to elements of $\MS$. Let $o$ be the point of $\MS$ corresponding to the zero vector.

By $B$ a set of positive roots $\RS^+\subset \RS$ is defined which determines the \emph{fundamental Weyl chamber} 
$$
\Cf\define\{x\in\MS : \langle x,\alpha^\vee\rangle \geq 0 \text{ for all } \alpha\in\RS^+\}
$$
with respect to $B$. By replacing some (which might be all or none) of the inequalities in the definition of $\Cf$ by equalities we obtain \emph{faces of the fundamental Weyl chamber}. 

The spherical Weyl group $\sW$ of $\RS$ acts by reflections $r_\alpha$, $\alpha\in \RS$, on the model space $\MS$.
The fixed point sets of the $r_\alpha$ are called \emph{hyperplanes} and are denoted by $H_\alpha$.
One has $H_\alpha=\{x \in \MS: r_\alpha(x)=x\}=\{x \in\MS: \langle x, \alpha^\vee\rangle =0\}$.

An \emph{affine Weyl group} is the semidirect product of a group of translations $T$ of $\MS$ by $\sW$. If $T$ equals $\MS$, then $\aW \define \sW\rtimes T$ is called the \emph{full affine Weyl group}.
The actions of $\sW$ and $T$ on $\MS$ induce an action of $\aW$ on $\MS$. 
An \emph{(affine) reflection} is an element of $\aW$ which is conjugate in $\aW$ to a reflection $r_\alpha$, $\alpha \in \RS$.
A \emph{hyperplane} $H_r$ in $\MS$ is the fixed point set of an affine reflection $r$. It determines two half-spaces of $\MS$ called \emph{half-apartments}. 



We define a \emph{Weyl chamber} in $\MS$ to be an image of a fundamental Weyl chamber under the affine Weyl group $\aW$. The image of the \emph{faces} of the fundamental Weyl chamber then define the faces of this Weyl chamber. A face of a Weyl chamber will also be called a \emph{Weyl simplex}. 
Note that a Weyl simplex $S$ contains exactly one point $x$ which is the intersection of all bounding hyperplanes of $S$. We call it the \emph{base point} of $S$ and say $S$ is \emph{based at $x$}.

Let $\Lambda$ be a totally ordered abelian group and let $X$ be a set. A metric on $X$ with values in $\Lambda$, short a \emph{$\Lambda$--valued metric}, is a map $d:X\times X \mapsto \Lambda$ satisfying the usual axioms of a metric. That is positivity, symmetry ($d(x,y)=d(y,x)$), equality $d(x,y)=0$ if and only if $x=y$ and the triangle inequality for arbitrary triples of points. The pair $(X,d)$ is called \emph{$\Lambda$--metric space}.

A particular $\aW$--invariant $\Lambda$--valued metric on the model space $\MS$ is defined by
$$
d(x,y)= \sum_{\alpha\in\RS^+} \vert \langle y-x, \alpha^\vee \rangle \vert.
$$

A subset $Y$ of $\MS$ is called \emph{convex} if it is the intersection of finitely many half-apartments. This includes the empty set and $\MS$. The \emph{convex hull} of a subset $Y\subset \MS$ is the intersection of all half-apartments containing $Y$.

Note that Weyl simplices and hyperplanes, as well as finite intersections of convex sets  are convex. A convex hull of a subset of the model space is not necessarily convex due to the finiteness requirement. 


\begin{definition}\label{Def_LambdaBuilding}
Let $X$ be a set and $\cA$ a collection of injective maps $f:\MS\hookrightarrow X$, called \emph{charts}.
The images $f(\MS)$  of charts $f\in\cA$ are called \emph{apartments} of $X$. Define \emph{Weyl chambers, hyperplanes, half-apartments, ... of $X$} to be images of such in $\MS$ under any $f\in\cA$. The set $X$ is a \emph{(generalized) affine building} with \emph{atlas} (or \emph{apartment system}) $\cA$ if the following conditions are satisfied
\begin{itemize}[label={(A*)}, leftmargin=*]
 \item[(A1)] If $f\in\cA$ and $w\in \aW$ then $f\circ w\in\cA$. 
 \item[(A2)] Let $f,g\in\cA$ be two charts. Then $f^{-1}(g(\MS))$ is a convex subset of $\MS$. There exists $w\in \aW$ with $f\vert_{f^{-1}(g(\MS))} = (g\circ w )\vert_{f^{-1}(g(\MS))}$. 
 \item[(A3)] For any two points in $X$ there is an apartment containing both.
 \item[(A4)] If $S_1$ and $S_2$ are two Weyl chambers in $X$ there exist sub-Weyl chambers $S_1' \subset S_1, S_2' \subset S_2$ in $X$ and an $f\in\cA$ such that $S_1'\cup S_2' \subset f(\MS)$. 
 \item[(A5)] For any apartment $A$ and all $x\in A$ there exists a \emph{retraction} $r_{A,x}:X\to A$ such that $r_{A,x}$ does not increase distances and $r^{-1}_{A,x}(x)=\{x\}$.
 \item[(A6)] If $f,g$ and $h$ are charts such that the associated apartments intersect pairwise in half-apartments then $f(\MS)\cap g(\MS)\cap h(\MS)\neq \emptyset$. 
\end{itemize}
The \emph{dimension} of the building $X$ is $n=\mathrm{rank}(\RS)$, where $\MS\cong \Lambda^n$. 
\end{definition}

Conditions (A1)-(A3) imply the existence of a $\Lambda$--distance on $X$, that is a function $d:X\times X\mapsto \Lambda$ satisfying all conditions of the definition of a $\Lambda$--metric but the triangle inequality.  Given $x,y$ in $X$ fix an apartment containing $x$ and $y$ with chart $f\in\cA$ and let $x',y'$ in $\MS$ be defined by $f(x')=x, f(y')=y$. 
The \emph{distance} $d(x,y)$ between $x$ and $y$ in $X$ is given by $d(x',y')$. By Condition (A2) this is a well-defined function on $X$. Therefore it makes sense to talk about a distance non-increasing function in (A5). Note further that, by (A5), the defined  distance function $d$  satisfies the triangle inequality. Hence $d$ is a metric on $X$.

In the case of $\R$--buildings one has that Condition (A6) follows from the other conditions. This can be found (along with other equivalent definitions for this particular case) in~\cite{Parreau}. One can also define $\R$--buildings in a more geometric way, see~\cite{KleinerLeeb}. There is a paper in preparation by the first author investigating alternative definitions for generalized affine buildings.


\subsection{Local and global structure of $\Lambda$--affine buildings}\label{section:lg}

There are two types of spherical buildings associated to an affine $\Lambda$--building $(X,\cA)$ of type $\MS(\RS, \Lambda)$:  the spherical building $\partial_{\cA}X$ at infinity and at each point $x\in X$ a so-called residue $\Delta_xX$. 

Two subsets $\Omega_1,\Omega_2$ of a $\Lambda$--metric space are \emph{parallel} if there exists $N\in\Lambda$ such that for all $x\in \Omega_i$ there exists an $y\in \Omega_j$ such that $d(x,y)\leq N$ for $\{i,j\}=\{1,2\}$.
Note that parallelism is an equivalence relation. One can prove

\begin{prop}{\cite[Section 2.4]{Bennett}}
Let $\MS=\MS(\RS, \Lambda)$ be the model space equipped with the full affine Weyl group $\aW$. Then the following is true.
\begin{enumerate}
\item Two hyperplanes (or two Weyl simplices) are parallel if and only if they are translates of each other by elements of $\aW$.
\item For any two parallel Weyl chambers $S$ and $S'$ there exists a Weyl chamber $S''$ contained in $S\cap S'$ and parallel to both. 
\end{enumerate}
\end{prop}

A simplex in the \emph{spherical building at infinity} is a parallel class $\partial S$ of a Weyl simplex $S$ in $X$. Hence as a set of simplices
$$
\partial_\cA X =\{\partial S : S \text{ is a Weyl simplex of } X \}.
$$
One simplex $\partial S_1$ is contained in a simplex $\partial S_2$ if there exist representatives $S'_1, S'_2$ which are contained in a common apartment with chart in $\cA$, having the same base point and such that $S'_1$ is contained in $S'_2$.

\begin{prop}
The set $\partial_\cA X$ defined above is a spherical building of type $\RS$ with apartments in a one-to-one correspondence with apartments of $X$.
\end{prop}
\begin{proof}
See \cite[3.6]{BennettThesis} or \cite[5.7]{PetraThesis}.
\end{proof}

To define a second type of equivalence relation on Weyl simplices we say that two of them, $S$ and $S'$, \emph{share the same germ} if both are based at the same point and if $S\cap S'$ is a neighborhood of $x$ in $S$ and in $S'$. It is easy to see that this is an equivalence relation on the set of Weyl simplices based at a given point. The equivalence class of $S$, based at $x$, is denoted by $\Delta_x S$ and is called \emph{germ of $S$ at $x$}.
The germs of Weyl simplices based at a point $x$ are partially ordered by inclusion: $\Delta_x S_1\subset \Delta_xS_2$ if there exist representatives $S'_1, S'_2$ contained in a common apartment such that $S_1'$ is a face of $S_2'$. Let $\Delta_xX$ be the set of all germs of Weyl simplices based at $x$. Then

\begin{prop}{\cite[5.17]{PetraThesis}}\label{prop:exis_res}
For all $x\in X$ the set $\Delta_xX$ is a spherical building of type $\RS$ which is independent of $\cA$.
\end{prop}

Let $\mu$ be a germ of a Weyl simplex $S$ based at $x$. We say that \emph{$\mu$ is contained in a set $K$ } if there exists an $\varepsilon > 0$ in $\Lambda$ such that $B_\varepsilon(x) \cap S$ is contained in $K$.

The following properties will be of use in subsequent proofs of the present paper.

\begin{prop}\label{prop:list}
Let $(X,\cA)$ be an affine building of type $\MS(\RS,\Lambda)$. Then: 
\begin{enumerate}
  \item \label{list:5.23} Let $S$ and $T$ be two Weyl chambers based at the same point $x$. If their germs are opposite in $\Delta_xX$ then there exists a unique apartment containing $S$ and $T$.
  \item \label{list:tec18} For any germ $\mu\in X$ the affine building $X$ is, as a set, the union of all apartments containing $\mu$.
\end{enumerate}
\end{prop}
\begin{proof}
See 5.23 and 5.13 of \cite{PetraThesis} for a proof.
\end{proof}

The proof of the following proposition is the same as of Proposition 1.8 in \cite{Parreau}. A consequence of it is that given a point in $X$ and parallel class of Weyl simplices, there is a unique Weyl simplex in this class based at the given point.

\begin{prop}\label{prop:tec}
Let $(X,\cA)$ be an affine building and $c$ a chamber in $\partial_\cA X$. For a Weyl chamber $S$ based at a point $x\in X$ there exists an apartment $A$ with chart in $\cA$ containing a germ of $S$ at $x$ and such that $c$ is contained in the boundary $\partial A$.
\end{prop}

Given a germ of a Weyl chamber in a fixed apartment $A$ one can define a retraction of the building onto $A$ as follows.

\begin{definition}\label{Def_vertexRetraction}
Fix a germ $\mu$ of a Weyl chamber in $X$.
Given a point $x$ in $X$ let $g$ be a chart in $\cA$ such that $x$ and $\mu$ are contained in  $g(\MS)$. Define 
$$
 r_{A,\mu}(x) \define (f\circ w\circ g^{-1} )(x)  
$$
where $w\in\aW$ is such that $g\vert_{g^{-1}(f(\MS))}=(f\circ w)\vert_{g^{-1}(f(\MS))}$.
The map $r_{A,\mu}$ is called the \emph{retraction onto $A$ centered at $\mu$}.
\end{definition}

By Condition (A2) and item \ref{list:tec18} of Proposition~\ref{prop:list} this retraction is well-defined. Furthermore, as proved in Appendix C of \cite{PetraThesis}, it is distance non-increasing. Furthermore, the restriction of $r_{A,\mu}$ to an apartment containing $\mu$ is an isomorphism onto $A$.

We end these preliminaries by pointing out that our main results (in particular Main Result~\ref{theorem:secondmain}, part~\ref{item:mr22}) allow for more spherical buildings to be defined from a $\Lambda$--building than the two constructions mentioned in this section. In fact, one can associate a spherical building to each set of points with distance in a convex subgroup $\Lambda'$ of $\Lambda$ from a certain point of $X$. The spherical building at infinity and the residues correspond to the choices $\Lambda'=\Lambda$ and $\Lambda' = \{0\}$.


\section{Proof of the first Main Result}\label{section:firstMain}

Given an epimorphism $e:\Lambda\to\Lambda'$ of ordered abelian $\Q [\{\alpha^\vee(\beta)\}_{\alpha,\beta\in\RS}\}]$--modules $\Lambda$ and $\Lambda'$, we define the \emph{base change functor} as follows.
Let $(X,\cA)$ be an affine building with model space $\MS(\RS,\Lambda)$ (or shortly $\MS$) and distance function $d$ which is induced by the standard distance on the model space.

The relation ``$\sim$'' on $X$ with $x\sim y$ when $d(x,y)\in\ker(e)$ is an equivalence relation (due to the triangle inequality). Let $X'$ be the quotient of $X$ defined by this equivalence relation. The associated quotient map $\phi:X\to X'$ is surjective by definition. One can define a metric $d'$ on $X'$ by putting $d'(\phi(x),\phi(y))\define e(d(x,y))$. This metric is well-defined due to the triangle inequality, one also easily checks it is indeed a metric. Let $\MS'$ be the model space $\MS(\RS,\Lambda')$ and $W'$ the associated affine Weyl group. In the same way as for $X$ one can define a map $\phi_\MS$ from the model space $\MS$ to $\MS'$. For each chart $f \in \MS$ one has that the preimages of the maps $\phi \circ f$ and $\phi_\MS$ on $\MS$ are the same. Hence one can define an injective map $f': \MS \to X'$ such that $\phi \circ f$ equals $f' \circ \phi_\MS$.


This way we have defined a set of charts $\cA'$ from $\MS'$ into $X'$. Automatically we also have defined (half-)apartments, hyperplanes, Weyl chambers, \dots~in $X'$. By construction these objects are the images under $\phi$ of similar objects in $X$.

Conditions (A1) and (A3)-(A5) for $(X',\cA')$ are easy consequences of the fact that these conditions are already satisfied by $(X,\cA)$. The only non-trivial condition to check is Condition (A2). This turns out to be particularly difficult when two non-intersecting apartments intersect after applying $\phi$. Condition (A6) follows as a byproduct of the proof of (A2).

The outline of the proof is the following. We start with investigating images of pairs of already intersecting convex sets (Lemma~\ref{lemma:compare}). This will imply that for two already intersecting apartments nothing surprising happens (Lemma~\ref{cor:A2part}). 

The next step is then to investigate local structures, i.e. the residues. The easier case of already intersecting apartments will be sufficient to show that these local structures are spherical buildings (Lemma~\ref{lemma:spheric}). Condition (A6) follows from this case as well. The results we obtain in this part are also useful to prove the second Main Result later on.  

This local information eventually allows us to prove (A2) in full generality (Lemma~\ref{lemma:a2general}). After this we end by showing functoriality.

\subsection{Intersecting convex sets}

In this section we study how already intersecting convex sets behave under the map $\phi$. These lemmas will be used later on to investigate the local structure of the quotient space $(X', \cA')$ and in the proof of Condition (A2).

\begin{lemma}\label{lemma:inker}
Let $x$ and $x'$ be two points of the model space $\MS$ lying in respectively two Weyl simplices $S$ and $S'$ both based at some point $y$. Suppose that $d(x,x') \in \ker e$. Then if $S$ and $S'$ do not have Weyl simplices in common, other than the base point, one has that $d(x,y),d(x',y) \in \ker e$.
\end{lemma}
\proof
The images of the two Weyl simplices $S$ and $S'$ under $\phi_\MS$ are again two Weyl simplices $\phi_\MS(S)$ and $\phi_\MS(S')$ having no common Weyl simplices. So the intersection of $\phi_\MS(S)$ and $\phi_\MS(S')$ is the singleton $\{\phi_\MS(y)\}$. As the point $\phi_\MS(x)=\phi_\MS(x')$ lies in this intersection, one has that $\phi_\MS(x)=\phi_\MS(x') = \phi_\MS(y)$. By the definition of $\phi_\MS$, we conclude that $d(x,y),d(x',y) \in \ker e$.
\qed

\begin{lemma}\label{lemma:germ}
Suppose that some subset $K$ of the model space $\A$ is closed under taking convex hulls of pairs of points of $K$. Then a germ based at a point $k\in K$ lies in $K$, if and only if, there is a point $x\in K$ contained in the Weyl simplex $S$ in $\A$ corresponding to that germ, and $S$ is the minimal Weyl simplex containing $x$.
\end{lemma}
\proof
Let $k$ be a point of $K$ and $\mu$ a germ of a Weyl simplex $S$ based at $k$. We have to prove that there exists an $\varepsilon>0$ such that $B_\varepsilon(k)\cap S\subset K$ if and only if there is a point $x \in K$ lying in the Weyl simplex $S$ corresponding to that germ, such that $S$ is the minimal Weyl simplex containing $x$.

First assume that there is such a point $x$ in $K$. Consider the minimal Weyl simplex $T$ based at $x$ containing $k$. One has that $\partial T$ is opposite $\partial S$ in $\partial \A$. By assumption, the convex hull of $x$ and $k$ is contained in $K$. But since this convex hull is the intersection of $S$ and $T$ we have that there is an $\varepsilon \leq d(x,k)$ such that $S\cap B_\varepsilon(k)$ is contained in $K$.

Conversely assume that $\mu$ is a germ of a Weyl simplex $S$ based at $k$ contained in $K$. So there exists an $\varepsilon>0$ such that $S\cap B_\varepsilon(k)$ is contained in $K$. If we can prove that there exists a point $x$ in $S$ but not on a non-maximal face of this Weyl simplex, with distance less than $\varepsilon$ to $k$, we are done.

Let $\RS$, $B$ and $F$ be as in Section~\ref{section:defs}. Consider the submodule $M$ of $\Lambda$ spanned by $\varepsilon$, this submodule is isomorphic to $F$. Consider only the linear combinations $\sum_{\alpha\in B} v_\alpha\alpha$, with $v_\alpha \in M$, as points (see Section~\ref{section:defs}). This way, the problem is reduced to the case where $\Lambda$ is isomorphic to $F$, a subfield of the reals. Assume we are in this case, let $y$ be a point of $S$, not on a non-maximal face of the Weyl simplex. Due to the field nature of $\Lambda$ one can now find an $f \in F$ such that the product of $f$ with distance $d(o,y)$ is $\varepsilon/2$. Taking the scalar product of $f$ with the vector corresponding to $y$, one obtains a vector corresponding to a point $x$ with the same properties as $y$ but at distance $\varepsilon/2$ from $o$. This concludes the proof.
\qed



\begin{lemma}\label{lemma:conv_exist}
Let $K$ be a convex subset of the model space $\MS$, and $x$ a point of $\A\!\setminus\! K$. Then there exists a point $y \in K$ and Weyl simplex $S$ based at $y$, containing $x$, such that the intersection $S \cap K$ is exactly $\{y\}$.
\end{lemma}
\proof
Let $S$ be a Weyl simplex based at $x$ of minimal dimension while having a non-empty intersection with $K$. By the minimality no face of $S$ contains points of $K$ (except $S$ itself). Consider for each point $k$ in the intersection $S \cap K$ the Weyl simplex $S'_k$ based at $k$ and in the opposite direction of $S$. It follows that each such $S'_k$ contains $x$.

Let $y$ be a point of $S \cap K$ such that the face of the germ of $S'_{y}$ in $S \cap K$ is minimal. Suppose that $S'_y \cap S \cap K$ contains more than just the point $y$. If this is the case we can find some one-dimensional face $R$ of the Weyl simplex $S'_y$ such that $R \cap S\cap K$ contains more than just $y$. It is impossible that $R$ lies completely in $K$ as no (non-maximal) face of $S$ contains points of $K$ (while $R$ will contain a point of such a face of $S$). So $R \cap K$ is a line segment bounded two endpoints, one is $y$, call the other $y'$ (these endpoints exist by convexity). 

Consider the Weyl simplex $S'_{y'}$. By the previous lemma one has that if a face of the germ of $S'_{y'}$ lies in $S \cap K$, then the corresponding face of the germ $S'_y$ based at $y$ obtained by translation will also be contained in $S \cap K$. Moreover the face of the germ of $S'_{y}$ in $S \cap K$ is strictly larger than the one of $S'_{y'}$, this because the germ of $R$ is also contained in $S'_{y} \cap S \cap K$. This violates minimality. It follows that $S'_y \cap S \cap K=\{y\}$. 

As no (non-maximal) face of $S$ contains points of $K$, it follows from the convexity of $K$ that $S \cap K = \pi \cap K$ with $\pi$ the subspace (of the model space) spanned by $S$. As $S'_y$ lies in $\pi$, this implies that $S'_y \cap K = S'_y \cap \pi \cap K  = S'_y \cap S \cap K=  \{y\} $. We conclude that the point $y$ and Weyl simplex $S'_y$ have the desired properties.
\qed

\begin{lemma}\label{lemma:compare}
Let $K$ and $K'$ be two convex subsets of respectively apartments $A$ and $A'$ of the affine building $(X,\cA)$, such that their intersection $K \cap K'$ is non-empty. If $x \in K$ and $x' \in K'$ are two points with $d(x,x') \in \ker e$, then there exists a point $z \in K \cap K'$ with $d(x,z),d(x',z) \in \ker e$.  
\end{lemma}
\proof
Note that if we find a point $z \in K \cap K'$ such that $d(x,z) \in \ker e$, then also $d(x',z) \in \ker e$ holds by the triangle inequality. In order to exclude trivialities, we assume that $x,x' \notin K \cap K'$.

Let $K \cap K'$ and $x$ be the convex subset and point on which we apply Lemma~\ref{lemma:conv_exist}.  This lemma yields us a point $a$ of $K \cap K'$. Let $S$ and $S'$ be the minimal Weyl simplices in respectively $A$ and $A'$, both based at $a$ and containing respectively $x$ and $x'$. The germs of $S$ and $S'$ (which lie in respectively $K$ and $K'$ due to minimality and convexity) have no simplices in common, because otherwise the intersection of the Weyl simplex $S$ with $K$ and $K'$ would contain more than just $a$, contradicting the construction of $a$. 

Let $R$ be a Weyl chamber in $A$ based at $a$ and containing $S$. Consider the image  $R'$ of $S'$ under the retraction $r_{A,R}$ (see Definition~\ref{Def_vertexRetraction}). As the germs of $S$ and $S'$ have no simplices in common, the Weyl simplices $S$ and $R'$ do not have Weyl simplices in common either. Let $y$ be the image of $x'$ under the retraction. As distances are not increased by this, we have that $d(x,y) \in \ker e$.  By Lemma~\ref{lemma:inker} and the previous discussion, we obtain that $d(x,a) \in \ker e$. Putting $z = a$ finishes the proof. 
\qed

\begin{lemma}\label{cor:sectors1}
Let $S$ and $S'$ be two Weyl simplices of the $\Lambda$--building $(X,\cA)$ having non-empty intersection. If $x \in S$ and $x' \in S'$ are such that $d(x,x') \in \ker e$, then there exists a point $z \in S \cap S'$ such that $d(x,z),d(x',z) \in \ker e$.  
\end{lemma}
\proof
Directly from the previous lemma.
\qed

The following lemma shows that if the base points of two parallel Weyl chambers are close together, that then the entire Weyl chambers are close together.

\begin{lemma}\label{cor:sectors2}
Let $S$ and $S'$ be two parallel Weyl chambers of the $\Lambda$--building $(X,\cA)$, based at respectively $y$ and $y'$ with $d(y,y') \in \ker e$. Then there exists an isometry $\tau$ from $S$ to $S'$, such that if $x\in S$, then $d(x,\tau(x)) \in \ker e$.
\end{lemma} 
\proof
As the intersection of parallel Weyl chambers is non-empty, the previous lemma shows the existence of a point $z$ in $S\cap S'$ such that $d(y,z),d(y',z) \in \ker e$. Consider the Weyl chamber $S''$ based at $z$ parallel to $S$ and $S'$. It is a sub-Weyl chamber of both $S$ and $S'$.

Let $A$ be an apartment containing $S$. There exists a translation of $A$ mapping $y$ to $z$ (and so also $S$ to $S''$). Each point of $A$ is mapped to another point of $A$ with the distance between both in $\ker e$. In the same way as we have defined an isometry from $S$ to $S''$ here, we can define an isometry from $S''$ to $S'$. Combining these two isometries yields the desired isometry.
\qed

The last result of this section shows that for already intersecting apartments nothing surprising happens.

\begin{lemma}\label{cor:A2part}
If two apartments $A$ and $B$ of $(X,\cA)$ already intersect before applying $\phi$, then $\phi(A) \cap \phi(B) =\phi(A \cap B)$ and (A2) will be satisfied for each pair of charts of the apartments $\phi(A)$ and $\phi(B)$ of $(X',\cA')$.
\end{lemma}
\proof
It is clear that $\phi(A \cap B)$ is a convex subset of $\phi(A)$, as the images of the finite number of half-apartments of $A$ with $A \cap B$ as intersection, will be a finite number of half-apartments of $\phi(A)$ with $\phi(A\cap B)$ as intersection.

Let $x \in A$ be a point such that $\phi(x) \in \phi(B)$. So there exists a point $x' \in B$ such that $\phi(x)=\phi(x')$, or equivalently $d(x,x') \in \ker e$. Lemma~\ref{lemma:compare} implies that there is a point $z \in A \cap B$ with $d(x,z),d(x',z) \in \ker e$.  So $\phi(x)=\phi(x')=\phi(z)$. We can conclude that $\phi(A) \cap \phi(B) =\phi(A \cap B)$, which is a convex set.

The second part of (A2), this being the existence of a $w'\in W'$ with certain properties, follows directly from Condition (A2) for the original building $(X,\cA)$.
\qed

\subsection{Local structures of $(X',\cA')$}\label{section:locs}

Although we did not prove Condition (A2) yet, one can define germs in $(X',\cA')$ in the same way as in Section~\ref{section:lg}. The set of germs based at some point $\phi(x) \in X'$ forms again a simplicial complex $\Delta_{\phi(x)} X'$. The goal of this section is to prove that $\Delta_{\phi(x)} X'$ is again a spherical building. Denote by $X''$ all the points in $X$ such that the distance to the point $x$ lies in $\ker{e}$. The germs of Weyl chambers at $\phi(x)$ are the chambers of $\Delta_{\phi(x)}X'$.

\begin{lemma}\label{lemma:inap}
Two germs at $\phi(x)$ lie in a common apartment of $(X',\cA')$.
\end{lemma}
\proof
This follows directly from a proposition with the same statement and proof of Proposition~\ref{prop:tec}.
\qed

\begin{lemma}\label{lemma:germs}
Two Weyl chambers $S_1$ and $S_2$ in $(X,\cA)$ based at respectively $x_1$ and $x_2$, such that $\phi(x_1)=\phi(x_2)=\phi(x)$, give rise to the same germ $\Delta_{\phi(x)} \phi(S_1) = \Delta_{\phi(x)} \phi(S_2)$ in $X'$, if and only if, there exists a point $x' \in X'' \cap S_1 \cap S_2$ such that the sub-Weyl chambers of $S_1$ and $S_2$ based at $x'$ are identical restricted to $X''$.
\end{lemma}
\proof
Assume that $S_1$ and $S_2$, with properties as stated in the lemma, give rise to the same germ $\Delta_{\phi(x)} \phi(S_1) = \Delta_{\phi(x)} \phi(S_2)$. We first deal with the case where the two sectors $S_1$ and $S_2$ have a common point $z \in X''$.


Apply Lemma~\ref{cor:sectors2} to the pairs $S_i$, $S_i'$ to obtain that $\Delta_{\phi(x)}\phi(S_i')=\Delta_{\phi(x)}\phi(S_i)$, so $\Delta_{\phi(x)}\phi(S_1') = \Delta_{\phi(x)}\phi(S_2')$. This makes it possible to find points $z_1$ of $S_1'$ and $z_2$ of $S_2'$ such that $\phi(z_1)=\phi(z_2)$, and such that this image does not lie on the (non-maximal) faces of the Weyl chambers $\phi(S_1)$ and $\phi(S_2)$. Applying Lemma~\ref{lemma:compare} to $S_1'$ and $S_2'$ and obtain a point $z' \in S_1'\cap S_2'$ such that $\phi(z')=\phi(z_1)=\phi(z_2)$. The convex hull of $z$ and $z'$ lies in both $S_1'$ and $S_2'$, and contains both $S_1' \cap X''$ and $S_2' \cap X''$, this proves that $S_1' \cap X'' = S_2' \cap X''$. 

Now we show we can always find such a $z$. Let $S_1''$ be the Weyl chamber based at $x_2$ and parallel to $S_1$. Lemma~\ref{cor:sectors2} one again has that $\Delta_{\phi(x)} \phi(S_1) = \Delta_{\phi(x)} \phi(S_1'')$. So also $\Delta_{\phi(x)} \phi(S_1'') = \Delta_{\phi(x)} \phi(S_2)$, and as $S_1''$ and $S_2$ clearly intersect, we can apply the previous case and conclude that $S_1'' \cap X'' = S_2 \cap X''$. Combining this with the implication of Lemma~\ref{cor:sectors1} that the Weyl chambers $S_1$ and $S_1''$ intersect in $X''$, we have that $S_1$ and $S_2$ have a point in common in $X''$.

We now prove the other direction. Let $F_1$ and $F_2$ be two one-dimensional Weyl simplices of the same type, with source $x'$, and assume that they are faces of the Weyl chambers parallel to respectively $S_1$ and $S_2$ with source $x'$. The intersection $F_1 \cup F_2$ is convex in $F_1$, thus either $F_1=F_2$, or there exists a point $a \in F_1$ such that $F_1 \cap F_2$ equals the convex hull of $x'$ and $a$. As $F_1 \cap X'' \subset F_1 \cap F_2$, the distance between $x'$ and $a$ does not lie in $\ker e$. This implies that $\Delta_{\phi(x)}\phi(F_1) = \Delta_{\phi(x)}\phi(F_2)$. Repeating this argument proves eventually that $\Delta_{\phi(x)}\phi(S_1) = \Delta_{\phi(x)}\phi(S_2)$.
\qed


\begin{lemma}
If the intersection $\phi(A) \cap \phi(B)$ of the images of two apartments $A$ and $B$ contains a germ of a Weyl chamber, then the intersection $A \cap B$ is non-empty. 
\end{lemma}
\proof
By applying the previous lemma to two Weyl chambers in $A$ and $B$ such that their images have the same germs.
\qed

Note that the above lemma combined with Lemma~\ref{cor:A2part} implies that if two apartments are identical after the base change functor, then they are also identical before it. Once we prove that $(X',\cA')$ is indeed a $\Lambda'$--building, it follows that the buildings at infinity of $(X,\cA)$ and $(X',\cA')$ are isomorphic.

Another consequence is that if two apartments intersect in a half-apartment after the base change functor, they also intersected in one before it. As the $\Lambda$--building $(X,\cA)$ satisfies Condition (A6), it follows that $(X',\cA')$ satisfies it too.

\begin{lemma}\label{lemma:spheric}
The set of germs $\Delta_{\phi(x)} X'$ forms a spherical building.
\end{lemma}
\proof
The germs in $\Delta_{\phi(x)} X'$ form a simplicial complex. Any two simplices lie in an apartment due to Lemma~\ref{lemma:inap}. The second property to check is if one has two apartments $\phi(A)$ and $\phi(B)$ containing $\phi(x)$, that when the corresponding apartments $\Delta_{\phi(x)} \phi(A)$ and $\Delta_{\phi(x)} \phi(B)$ share a chamber $C$ and simplex $D$, then there exists an isomorphism from $\Delta_{\phi(x)} A$ to $\Delta_{\phi(x)} B$, preserving $C$ and $D$ point-wise. The existence of this isomorphism now follows from Lemma~\ref{cor:A2part}, and the fact that $A$ and $B$ intersect due to the previous lemma.
\qed

\subsection{Proof of Condition (A2)}

To prove the first Main Result the only non-trivial condition to check is (A2). In the following let $A$ and $B$ be two apartments of $(X,\cA)$ such that the intersection $K :=\phi(A) \cap \phi(B)$ is non-empty. The method of the proof will consist of ``slightly shifting'' Weyl chambers using Lemma~\ref{cor:sectors2}, this to reduce to the easier case of already intersecting convex sets (see Lemma~\ref{cor:A2part}). The results of Section~\ref{section:locs} then allow us to ``glue'' this information together.

\begin{lemma}
Suppose $x \in A$ is a point such that $\phi(x)\in K$. Let $S$ be a Weyl chamber of $A$ based at $x$. Then there is a convex set $K_{S,x} \subset S$, containing $x$ and such that $\phi(K_{S,x}) = K \cap \phi(S)$. 
\end{lemma}
\proof
As $\phi(x) \in K$ there exists a point $x' \in B$ such that $\phi(x)=\phi(x')$. Let $S'$ be the Weyl chamber parallel to $S$ and based at $x'$. Lemma~\ref{cor:sectors2} defines a isometry $\tau$ from $S'$ to $S$. The intersection $K'$ of $S'$ with $B$ is convex. Hence its image $K_{S,x}$ under $\tau$ will be a convex subset of $S$ and $A$. 

If $y$ is an element of $K_{S,x}$, then its preimage $y'$ under $\tau$ will by construction be in $B$. Lemma~\ref{cor:sectors2} tells us that $d(y,y') =d(\tau(y'),y') \in \ker e$, so $\phi(y)=\phi(y') \in K \cap \phi(S)$. This implies that $\phi(K_{S,x})$ is contained in $K \cap \phi(S)$.

Now let $z$ be a point of $S$ such that $\phi(z) \in K$. This implies that there exists a $z' \in B$ with $\phi(z)=\phi(z')$. Let $z'' \in S'$ be the preimage of $z$ under $\tau$. Note that the triangle inequality implies that $d(z',z'') \in \ker e$. As both $B$ and $S'$ are convex subsets of apartments, we can apply Lemma~\ref{lemma:compare} to find a point $a$ in $K'=B\cap S'$ having distances in $\ker e$ to both $z'$ and $z''$. So we conclude that $\phi(\tau(a)) = \phi(z') =\phi(z)$, and thus $\phi(K_{S,x}) \supset K \cap \phi(S)$, or by combining both inequalities $\phi(K_{S,x}) = K \cap \phi(S)$. 
\qed

\begin{lemma}
The set $K$ is convex in $\phi(A)$.
\end{lemma}
\proof
Let $x$ be a point of $A$ such that $\phi(x) \in K$. For each Weyl chamber $S_i$ based at $x$ (with $i \in \{1,\dots,n\}$, $n$ being the number of Weyl chambers in $A$ based at $x$), we can define a convex subset $K_i := K_{S_i,x}$ of $S_i$ using the previous lemma. We have that $\phi(\bigcup_i K_i) = K$.

We label the roots in $\RS$ by $r_1, \dots, r_m$, with $m=|\RS |$. For the convex subset $K_i$, let $K_i^j$ be the minimal positive half-apartment defined by the root $r_j$ containing $K_i$. That is $K$ is a set $H_{i,k_i}^+\define \{x : \langle x, r_i^\vee\rangle \geq k_i\}$ with $k_i\in\Lambda\cup \{-\infty\}$ as small as possible. For the purpose of the proof, we interpret $H_{i, -\infty}^+$, which is the whole apartment, also as a half-apartment. As each $K_i$ is convex, it follows easily that $K_i = \bigcap_j K_i^j$.

For $j \in \{1, \dots , m\}$, define $K^j$ to be the maximal half-apartment corresponding to the root $r_j$ in the set $\{K_1^j, K_2^j, \dots, K_n^j\}$. 

Let $K'$ be the intersection $\bigcap_j \phi(K^j)$. It is clear that $K \subset K'$ as each $\phi(K_i)$ will be a subset of $K'$. Also, $K'$ will by construction be the intersection of all half-apartments of $\phi(A)$ containing $K$, and so $K'$ and $\phi(K^j)$ will be independent of the choice of $x$. If we show that $K=K'$, then the lemma is proven. Suppose that $k'$ is a point of $K'$ but not of $K$. 

Let $F$ be the minimal Weyl simplex based at $x$, containing $k'$. Then applying Lemma~\ref{lemma:conv_exist} to the set $F \cap K$ (which is convex by the previous lemma), one obtains a point $k$ and a minimal Weyl simplex $S$ in $\phi(A)$, containing $k'$ and based at $k$, such that $S\cap F \cap
K=\{k\}$. Note that again by the previous lemma the set $S \cap K$ is a convex subset of the
apartment $\phi(A)$, which certainly contains $k$, if it would contain more it
contradicts Lemma~\ref{lemma:germ}. One can conclude that we have constructed a point $k$ and Weyl simplex $S$ in $\phi(A)$, containing $k'$, based at $k$, such that $S \cap K = \{k\}$. 

Lemma~\ref{lemma:spheric} shows that $\Delta_k X'$ is a spherical building in which $\phi(A)$ and $\phi(B)$ define two apartments $\Delta_k \phi(A)$ and $\Delta_k \phi(B)$. The intersection of both is a convex subset of $\Delta_k \phi(A)$. By~\cite[Prop. 3.137]{AB} convex subsets of apartments are finite intersections of half-apartments. So there exists a half-apartment $\Delta_k H$ of $\Delta_k \phi(A)$ in the spherical building $\Delta_k X'$, which does not contain the chamber $\Delta_k \phi(S)$.  

Let $H$ be the half-apartment of $A$ corresponding to $\Delta_k H$. Using Lemma~\ref{lemma:germ}, it follows that this half-apartment contains $K$ completely, but does not contain $S$. So $k'$ does not lie in this half-apartment (by minimality of $S$). We have obtained a contradiction and proved convexity.
\qed

\begin{lemma}\label{lemma:a2general}
Given two charts $f,g\in\cA'$ there exists  a $w\in \aW'$ with $f\vert_{f^{-1}(g(\MS'))} = (g\circ w )\vert_{f^{-1}(g(\MS'))}$.
\end{lemma}
\proof
Denote the apartment defined by the chart $f$ by $A$, the one defined by the chart $g$ by $B$. First note that if $A$ and $B$ do not have points in common, there is nothing to prove. So suppose that there exists a point $k \in A \cap B$. Using Condition (A1) we can assume without loss of generality that $f(o)=g(o)=k$.

Consider the spherical building $\Delta_k X'$ (see Lemma~\ref{lemma:spheric}), of which $\Delta_k A$ and $\Delta_k B$ are apartments. The fundamental Weyl chamber of $\MS$ defines a certain Weyl chamber based at $k$ in the apartment $A$, and hence a chamber in $\Delta_k A$, denote this chamber by $C_A$. Similarly define $C_B$ in $\Delta_k B$. One of the axioms of spherical buildings tells us that there is an isomorphism $\sigma$ from the apartment $\Delta_k A$ to the apartment $\Delta_k B$ preserving the intersection of both. 

Let $w \in \sW$ be the unique element in the spherical Weyl group acting on $\Delta_k B$ which maps $C_B$ to $\sigma(C_A)$. Interpreting this element in the larger affine Weyl group $\aW'$ one sees that $h:= g\circ w \circ f^{-1}$ is an isometry from $A$ to $B$ mapping the germ corresponding to $C_A$ to $\sigma(C_A)$. Moreover, $h$ will fix the germs in $\Delta_k A \cap \Delta_k B$ due to the way we constructed $\sigma$ and $w$.

Showing that $f\vert_{f^{-1}(g(\MS'))} = (g\circ w )\vert_{f^{-1}(g(\MS'))}$ is equivalent to showing that if $k' \in A \cap B$ then $h(k')=k'$.  So let $k'$ be an arbitrary point in $A \cap B$ different from $k$. Let $S_A$ be the minimal Weyl simplex in $A$ based at $k$ and containing $k'$. The germ of this simplex is fixed, as it is in the intersection of $A$ and $B$ (see Lemma~\ref{lemma:germ}). It now easily follows from $h$ being an isometry that $k'$ is fixed.
\qed 


This completes the proof of \ref{item:mr11} of the first Main Result.  The remaining parts are done in the subsequent section.

\subsection{Functoriality}

It remains to prove assertion \ref{item:mr13}, i.e. functoriality, of the first Main Result.

Consider two $\Lambda$--buildings $(X_1,\cA_1)$ and $(X_2,\cA_2)$. By the first assertion one obtains two $\Lambda'$--buildings $(X_1',\cA_1')$ and $(X_2',\cA_2')$ and corresponding maps $\phi_1: X_1 \to X_1'$ and $\phi_2:X_2 \to X_2'$.  Suppose we have an isometric embedding $\psi$ from $X_1$ to $X_2$, then there is a unique isometric embedding $\phi':X_1' \to X_2'$ such that the following diagram commutes.
$$
\xymatrix{
	X _1\ar[r]^{\psi}\ar[d]_{\phi_1} 	& X_2 \ar[d]^{\phi_2}\\
	X'_1 \ar[r]^{\psi'} & X'_2  
}
$$

This because each preimage of a point of $X_1'$ under $\phi_1$ is a preimage of a unique point of $X_2'$ under $\phi_2 \circ \psi$. The uniqueness of this map yields functoriality, in particular it directly implies the following lemma.

\begin{lemma}
With assumptions and notation as in the first main result let $G$ be a group acting on $X$ by isometries, then $G$ acts on $X'$ by isometries and the map $\phi$ is $G$--equivariant. 
\end{lemma}

This completes the proof of the first Main Result.



\section{Proof of the second Main Result}\label{section:secondMain}
In  this section we tackle the second Main Result, where we prove that the fibers of the map discussed in the previous section are again generalized affine buildings. The set of apartments on such a fiber will be the non-empty intersections of apartments with this fiber. Conditions (A1)-(A3) and (A5) will be more or less direct consequences from the corresponding conditions for the original building. In order to prove conditions (A4) and (A6) we consider the structure at infinity of the fiber, making use of results in Section~\ref{section:locs}.

Let $(X,\cA)$ be an affine building with model space $\MS\define \MS(\RS, \Lambda)$. As always assume that $\Lambda$ admits an $F$--module structure and let $e:\Lambda\to \Lambda'$ be an epimorphism of ordered abelian $F$--modules. The kernel of $e$ is again an ordered abelian group admitting a natural $F$--module structure.

As defined at the beginning of Section~\ref{section:firstMain}, let $\phi$ be the base change functor associated to $e$. According to the first Main Result, the image of $\phi$ carries the structure of a $\Lambda'$--building $(X',\cA')$ with model space $\MS':=\MS(R,\Lambda')$.


Fix some point $x\in X$ for the remainder of the section. Define $X''$ to be the set $\phi^{-1}(\phi(x))$. This will be the set of points of the $\ker(e)$--building we want to construct. Denote by $o$ the point in $\MS$ corresponding to the zero vector. Let $\MS''$ be the points $y$ in $\MS$ such that $d(o,y) \in \ker e$. Using the coordinate description of $\MS$ it follows directly that $\MS''$ can be identified with $\MS(\R,\ker(e))$. The Weyl group $\aW'' = \sW T''$ of this model space, with $T''=T\cap \MS''$, can canonically be interpreted as a subgroup of $\aW$. Note that the elements of $W''$ are exactly those elements $w$ of $W$ such that $w.o \in \MS''$. 

Let $\widetilde{\cA}$ be the set of charts $f$ in $\cA$ such that $\phi(x) \in \phi(f(\MS''))$. As the charts in $\cA$ are isometries, it follows from the triangle inequality that for each $f \in \widetilde{\cA}$ one has $f(\MS'') = f(\MS)\cap X''$. From this we can define injections $\cA''=\{f\vert_{\MS''} : f\in\widetilde{\cA}\}$ from $\MS''$ into $X''$. We now claim that $(X'',\cA'')$ is a $\ker(e)$--building.


A first observation is that apartments in $X''$ will be intersections of $X''$ with apartments of $X$ containing a point of $X''$ (because of Condition (A1) and that for all $f \in \widetilde{\cA}$ one has that $ f(\MS'') = f(\MS)\cap X''$).

\begin{lemma}\label{lemma:MR2_1}
The pair $(X'',\cA'')$ satisfies Conditions (A1)-(A3) and (A5).
\end{lemma}
\proof
To prove (A1) let $f''$ be an element of $\cA''$ and $w\in W''$. By definition $f''$ is the restriction of a chart $f\in \widetilde{\cA}$. Condition (A1) applied to $(X,\cA)$ implies that $f\circ w\in\cA$. As $\phi(x) \in \phi(f(\MS''))=\phi((f\circ w)(\MS''))$, one has that $f \circ w \in \widetilde{\cA}$. So $f \circ w\vert_{\MS''} \in \cA''$. This chart equals $f'' \circ w$ over $\cA''$, so we have proven (A1). 


The first part of (A2) is easily seen to be true because the intersection of a convex set in $\MS$ with $\MS''$ is a convex set of $\MS''$ (as the intersection of a finite number of half-apartments of $\MS$ with $\MS''$ is the intersection of a finite number of half-apartments of $\MS''$). The second part follows if we can show that if one has $f \circ w (y) = g(y)$ with $f,g \in \widetilde{\cA}$, $w\in W$ and $y \in \MS''$, then $w \in W''$. If this was not the case then $w.y \notin \MS''$ and $(f \circ w) (y) \notin X''$ as $f(\MS'') = f(\MS)\cap X''$. But $g(y) \in X''$, so we can conclude that $w \in W''$ and that (A2) holds.

Condition (A3) follows from Condition (A3) for $(X,\cA)$ and the fact that apartments in $X''$ are the intersections of $X''$ with apartments of $X$ containing a point of $X''$.

The last condition we want to prove here is Condition (A5). If we are given a point $y$ and apartment $A''$ of $(X'',\cA'')$ containing this point, we know that there exists an apartment $A$ of $X$ such that $A\cap X'' = A''$. Let $r_{A,y}$  be a retraction with respect to $A$ and $y$ as implied in Condition (A5) for $(X,\cA)$. This retraction maps points of $X''$ to points of $ A \cap X'' = A''$  by the triangle inequality, and does not increase distances. Hence it follows that $(X'',\cA'')$ satisfies (A5).\qed


To finish the proof of the second Main Result we have to verify that (A4) and (A6) hold as well. This is done in the subsequent propositions.

\begin{prop}\label{prop:MR2_2}
The pair $(X'',\cA'')$ satisfies (A4).
\end{prop}
\proof
Let $S''$ and $T''$ be two Weyl chambers of $(X'',\cA'')$. These Weyl chambers are restrictions  to $X''$ of (not necessarily unique) Weyl chambers $S$ and $T$ of $(X,\cA)$. These two Weyl chambers give rise to chambers $\Delta_{\phi(x)}\phi(S)$ and $\Delta_{\phi(x)}\phi(T)$ of the spherical building $\Delta_{\phi(x)}X'$. Lemma~\ref{lemma:inap} implies that there exists an apartment $A'$ of $(X',\cA')$ containing both germs. By construction of $(X',\cA')$ there also exists an apartment $A$ of $(X,\cA)$ such that $\phi(A) = A'$. Let $S'$ and $T'$ be Weyl chambers of $A$, both based at some point $y \in X''$, for which the images correspond respectively to the germs $\Delta_{\phi(x)}\phi(S)$ and $\Delta_{\phi(x)}\phi(T)$.

The (non-empty) intersection $X'' \cap A$ is an apartment $A''$ of $(X'',\cA'')$. The Weyl chambers $S$ and $S'$ of $(X,\cA)$ both give rise to the same germ $\Delta_{\phi(x)}\phi(S)$ of $(X',\cA')$, so Lemma~\ref{lemma:germs} implies that the apartment $A''$ of $(X'',\cA'')$ contains a sub-Weyl chamber of $S''$. A similar argument asserts that it also contains a sub-Weyl chamber of $T''$. This proves (A4) for $(X'',\cA'')$. 
\qed

As the remaining Condition (A6) is not needed to define the spherical building of infinity, we can consider this  structure $\partial_{\cA''} X''$ for $(X'',\cA'')$. The next lemma shows that this spherical building is in fact the residue $\Delta_{\phi(x)}X'$.

\begin{lemma}\label{lemma:infinity}
There is a canonical isomorphism between $\partial_{\cA''} X''$ and $\Delta_{\phi(x)}X'$.
\end{lemma}
\proof 
Lemma~\ref{lemma:germs} implies a bijection between the chambers of these two spherical buildings. It is easily seen that this bijection preserves adjacency, hence it defines an isomorphism between the two buildings.
\qed

\begin{prop}\label{prop:MR2_3}
The pair $(X'',\cA'')$ satisfies (A6).
\end{prop}
\proof
Let $A_1,A_2,A_3$ be three apartments of $X''$ which pairwise intersect in half-apartments. The boundaries $\partial A_i$ intersect as well in half-apartments and correspond to apartments $a_i$ in the residue of $\phi(x)$ in $X'=\phi(X)$ by Lemma~\ref{lemma:infinity}. There are two cases: either a) there exists a half-apartment $\alpha$ in $\Delta_{\phi(x)}X'$ such that $a_i\cap a_j = \alpha$ for all $i\neq j$ with $\{i,j\}\subset \{1,2,3\}$ , or b) the intersections $a_i\cap a_j$ are distinct for all three pairs of elements of $\{1,2,3\}$, that is $a_1\cup a_2\cup a_3 = (a_1\cap a_2) \cup (a_1\cap a_3) \cup (a_2\cap a_3)$. 

In case b) choose an apartment $B_1'$ in $X'$ such that $\Delta_{\phi(x)}B_1' = a_1$. Let $c$ be a chamber in $(a_2\cap a_3)$ sharing a panel with $a_1$. Let $T$ be a Weyl chamber $T$ based at $\phi(x)$ with germ $c$ such that $\partial T$ shares a panel with $\partial B_1'$. Then there exist precisely two Weyl chambers $S_2, S_3$ in $B_1'$ such that $\partial S_i$ is opposite $\partial T$ in $\partial_{\cA'}X'$ for $i=2,3$. By construction $\Delta_{\phi(x)}S_i$ is opposite $c$ in $\Delta_{\phi(x)}X'$. Hence there exist, by item~\ref{list:5.23} of Proposition~\ref{prop:list}, unique apartments $B_2'$ and $B_3'$ of $(X',\cA')$ containing $T$ and $S_2$, respectively $S_3$. It is easily seen that the apartments $B_i'$ pairwise intersect in half-apartments. From the definition of apartments in $(X',\cA')$ if follows that there are three apartments $B_i$ such that $\phi(B_i) = B_i'$ for all $i=1,2,3$. These again intersect pairwise in half-apartments, see the comments on Condition (A6) in Section~\ref{section:locs}. We conclude that the intersection of all three $B_i$ is non-empty. Due to the construction there exists a point $y\in X''$ in this intersection.

Lemma~\ref{lemma:germs} implies  that the boundary of $B_i\cap X''$ in $X''$ is the same as the boundary of $A_i$, since their images under $\phi$ induce the same germs in $\Delta_{\phi(x)}X'$. 
Therefore $B_i\cap X'' = A_i$ for all $i$. 
The fact that $y$ is contained in $B_1\cap B_2\cap B_3$ implies that $A_1\cap A_2\cap A_3 = B_1\cap B_2\cap B_3\cap X''\neq \emptyset$, and that (A6) holds for $X''$ in case b).

Assume that we are in case a). Let $c$ be a chamber of the half-apartment $\alpha$. With this chamber there corresponds a parallel class of Weyl chambers of $(X'',\cA'')$. By applying Condition (A4) any two of these Weyl chambers contain (at least) a common sub-Weyl chamber. If one takes three Weyl chambers $S_1$, $S_2$ and $S_3$ in this parallel class respectively in apartments $A_1$, $A_2$ and $A_3$, then it follows that these three Weyl chambers have a common intersection, so also the apartments have a common intersection. This completes the proof of the proposition.
\qed




Main Result~\ref{theorem:secondmain} now follows from combining Lemma~\ref{lemma:MR2_1} and Propositions~\ref{prop:MR2_2} and~\ref{prop:MR2_3}.


\section{Proof of the third Main Result}\label{section:thirdMain}

In this section we discuss the case where $e: \Lambda \mapsto \Lambda' $ is a monomorphism. The strategy will be to add points to each apartment to give it a $\Lambda'$--structure. The first part of the proof deals with defining this extension rigourously. After this we check that the obtained structure is indeed the generalized affine building we want. 

The main difficulty lies into showing Condition (A3), i.e. that every two points lie in a common apartment. This is not surprising as the construction consists of adding new points. The condition will eventually follow from a covering result for pairs of apartments (see Proposition~\ref{prop:cover}).

Let $\RS$ be a root system and let $F$ be a subfield of $\R$ containing the set of evaluations of co-roots on roots. Let $e:\Lambda \mapsto \Lambda' $ be a monomorphism of ordered abelian groups $\Lambda$ and $\Lambda'$ both admitting an $F$--module structure.

Using this monomorphism one can define a natural embedding of the model space $\MS:=\MS(\RS,\Lambda)$ into $\MS':=\MS(\RS,\Lambda')$ as follows. 
Choose a basis $B$ of $\RS$, then each element $x$ of $\MS$ has a presentation as $x=\sum_{\alpha\in B} \lambda_\alpha \alpha$. Define $\iota: \MS\mapsto\MS'$ by
$$
x=\sum_{\alpha\in B} \lambda_\alpha \alpha \longmapsto \iota(x)=\sum_{\alpha\in B}e(\lambda_\alpha)\alpha.
$$
Since $e$ is a monomorphism $\iota$ is injective. Given points $x=\sum_{\alpha\in B} \lambda_\alpha \alpha$ and $y=\sum_{\alpha\in B} \mu_\alpha \alpha$ in $\MS$ we can calculate the distance $d'$ in $\MS'$ of the images under $\iota$ of  these two points of $\MS$:

\begin{align*}
d'(\iota(x),\iota(y)) 
	& = \sum_{\beta\in\RS^+} \vert\langle \iota(y)-\iota(x),\beta^\vee\rangle\vert \\
	& = \sum_{\beta\in\RS^+} \vert\langle \sum_{\alpha\in B} (e(\mu_\alpha) - e(\lambda_\alpha))\alpha,\beta^\vee\rangle\vert \\
	& = e( \sum_{\beta\in\RS^+} \vert \sum_{\alpha\in B} (\mu_\alpha-\lambda_\alpha) \langle \alpha,\beta^\vee\rangle \vert ) \\ 
	& = e(d(x,y)),\\
\end{align*}

hence $\iota$ is an embedding of $\MS$ into $\MS'$.
Similarly one can embed the affine Weyl group $\aW$ of $\MS$ into the affine Weyl group $\aW'$ of $\MS'$. Each element $w$ of $\aW$ is a product of a translation $t$ and an element $\overline{w}\in\sW$ of the spherical Weyl group. Since elements of the translation subgroup $T$ of $\aW$ (respectively $T'$ of $\aW'$) can canonically be identified with points in $\MS$, (respectively points in $\MS'$), we can define a group monomorphism $\hat{\iota}:\aW\to\aW'$ by putting $w=\overline{w}t \longmapsto \hat{\iota}(w)\define \overline{w}\iota(t)$. Hence $\aW$ can be naturally identified with a subgroup of $\aW'$.

Let $H$ be a half-apartment of $\MS$. If $\iota$ is not an isomorphism then the image of $H$ under it is not a half-apartment of $\MS'$, but one can easily find a half-apartment $H'$ of $\MS'$ such that $a\in H $ if and only if $\iota(a) \in H'$. This allows us to define a map $\tilde{\iota}$ from the half-apartments of $\MS$ to half-apartments of $\MS'$. As each convex subset of $\MS$ is the intersection of a finite number of half-apartments of $\MS$, one can extend the map $\tilde{\iota}$ to send convex subsets of $\MS$ to convex subsets of $\MS'$, such that for each point $a$ and convex subset $K$ of $\MS$: $a\in K \Leftrightarrow \iota(a) \in \tilde{\iota}(K)$. It also easily follows for a convex subset $K$ of $\MS$ and group element $w \in \aW$ that $\tilde\iota(w.K)= \hat{\iota}(w).\tilde{\iota}(K)$.

Let $(X,\cA)$ be an affine building with model space $\MS=\MS(\Lambda,\RS)$ and $\iota$ an embedding of $\MS$ into $\MS'$ induced by a monomorphism $e$ as above. Using a quotient construction on $\cA \times \MS'$ we will define a space $X'$ and a set of charts $\cA'$ in Definition~\ref{def:extension} which will turn out to be the generalized affine building whose existence is claimed in the third Main Result. 

Denote the set of pairs $\cA \times \MS'$ by $\widetilde{X}$. Let $(f,a)$ and $(g,b)$ be two pairs of $\widetilde{X}$. By Condition (A2) we know that the inverse image $Z\define f^{-1}(g(\MS))$ is a convex subset of $\MS$ and that there exists a $w \in \aW$ such that $f\vert_Z = g \circ w\vert_Z$. We now define $(f,a)$ and $(g,b)$ to be equivalent (denoted by $(f,a) \sim (g,b)$) if and only if $a \in \tilde{\iota}(Z)$ and $\hat{\iota}(w). a =b$.  This is independent of the choice of $w \in \aW$, as an other choice $w'$ would necessarily have the same action on $Z$, and so also the action of $\iota(w)$ and $\iota(w')$ on $\tilde{\iota}(Z)$ would be the same. 

A first goal is now to prove that ``$\sim$'' is an equivalence relation.


\begin{lemma}\label{lemma:ltrans}
Let $(f,a)$ and $(g,b)$ be two pairs of $\widetilde{X}$ and $v\in W$, then $(f,a)  \sim (g,b) \Leftrightarrow (f,a) \sim (g \circ v^{-1},\hat{\iota}(v).b)$.
\end{lemma}
\proof 
First of all notice that $Z=f^{-1}(g(\MS)) = f^{-1}((g\circ v^{-1})(\MS))$, and that if $w$ is an element of $W$ such that $f\vert_Z = g \circ w \vert_Z$, that then also $f\vert_Z = (g \circ v^{-1})\circ (v \circ w) \vert_Z$.

The condition for $(f,a)  \sim (g,b)$ to hold is that $a\in \tilde{\iota}(Z)$ and $\hat{\iota}(w).a=b$ with $w$ as above. On the other hand the condition for $(f,a) \sim (g \circ v^{-1},\hat{\iota}(v). b)$ is $a\in \tilde{\iota}(Z)$ and $\hat{\iota}(v \circ w).a=\hat{\iota}(v). b$. One easily sees that these conditions are the same.
\qed


\begin{lemma}\label{lemma:ltrans2}
Let $(f,a)$ and $(g,b)$ be two pairs of $\widetilde{X}$ and $w\in W$ such that $f \vert_Z = g \circ w \vert_Z$ with $Z=f^{-1}(g(\MS))$, then $(f,a)  \sim (g,b) \Leftrightarrow (f,a) \sim (g \circ w, a)$ and $\hat{\iota}(w). a =b$.
\end{lemma}
\proof
Directly from the previous lemma.
\qed

\begin{lemma}\label{lemma:inj}
The pairs $(f,a)$ and $ (f,b)$ are equivalent if and only if $a =b$.
\end{lemma}
\proof
Directly from the definition of the relation, and the fact that for the element in $\aW$ mentioned in this definition one can take the identity. \qed

\begin{lemma}
The relation ``$\sim$'' on $\widetilde{X}$ is an equivalence relation.
\end{lemma}
\proof
Reflexivity is clear from the previous lemma. We now proof symmetry. Given $(f,a)\sim(g,b)$ there exists a $w\in \aW$ such that $f\vert_Z = g\circ w\vert_Z$ where $Z\define f^{-1}(g(\MS))$. Further, $a$ is contained in the set $\tilde{\iota}(Z)$ and $\hat{\iota}(w). a = b$, so $b=\hat{\iota}(w)^{-1} . a$. It remains to prove that $b$ is contained in the set $\tilde{\iota}(Y)$ with $Y\define g^{-1}(f(\MS))$. 
By assumption $w.Z=Y$ and so $\hat{\iota}(w).\tilde\iota(Z)= \tilde{\iota}(Y)$. This implies that $b=\iota(w) . a$ is contained in $\tilde{\iota}(Y)$ and that the the relation is symmetric.

The last property to check is transitivity. Note that Lemma~\ref{lemma:ltrans} already shows a weak version of transitivity. Using this weak version and Lemma~\ref{lemma:ltrans2} we can assume that we have three pairs $(f,a)$, $(g,a)$ and $(h,a)$, with $(f,a) \sim (g,a)$ and $(f,a) \sim (h,a)$ such that $f\vert_Z =g\vert_Z$ and $f\vert_{Z'}=h\vert_{Z'}$  with $Z:= f^{-1}(g(\MS))$ and $Z':= f^{-1}(h(\MS))$. Consider the convex set $K := Z \cap Z'$, by taking restrictions to both $Z$ and $Z'$ it follows that $g\vert_K =h\vert_K$. In particular one observes that $K \subset Z'':= g^{-1}(h(\MS))$, and that $a\in \tilde{\iota}(K)\subset \tilde{\iota}(Z'')$. By Condition (A2) we know that there exists a $w\in W$ such that $g\vert_{Z''}=h\circ w \vert_{Z''}$, and so also $g\vert_{K}=h\circ w\vert_{K}$. Because $h$ and $g$ are injections it follows that $w$ leaves $K$ invariant. This also implies that $\hat{\iota}(w)$ leaves $\tilde{\iota}(K)$ invariant and that $\hat{\iota}(w).a =a$. We can conclude that $(g,a) \sim (h,a)$ and that~``$\sim$'' is a transitive relation. 
\qed

\begin{definition}\label{def:extension}
Let $X'= (\cA\times \MS')\diagup_{\sim}$ and define for each chart $f \in \cA$ a map $\psi(f)$ from $\MS'$ to  $X'$ as follows:
$$
\psi(f) = [a\longmapsto (f,a)\diagup_{\sim}].
$$
We write $\cA'$ for the set of maps $\psi(f)\circ w'$, with $f\in \cA$ and where $w'$ ranges over all $w'\in\aW'$.\end{definition}

We will prove that $(X',\cA')$ satisfies the assertion of the third Main Result. Note that Condition (A1) follows already from the definition and part 2 of the following lemma.

\begin{lemma}\label{lemma:prepphi}
The elements of $\cA'$ satisfy:
\begin{enumerate}
 \item each map $f'\in\cA'$ is injective,
 \item if $w\in\aW$ and $f \in \cA$, then $\psi(f\circ w) = \psi(f) \circ \hat\iota(w)$. 
\end{enumerate}
\end{lemma}
\proof
The first property is a consequence of Lemma~\ref{lemma:inj}. To prove the second observe that
$$
\psi(f\circ w) = [a\mapsto (f\circ w, a)\diagup_{\sim}]
$$ 
and that
$$
\psi(f)\circ \hat\iota(w) = [a\mapsto (f,\hat\iota(w).a)\diagup_{\sim}].
$$
It therefore remains to prove that $(f,\hat\iota(w).a)$ and $(f\circ w, a)$ are equivalent, which follows from Lemma~\ref{lemma:ltrans}.
\qed

Let $x$ be a point of $X$ lying in some apartment $\Sigma$. Let $f \in \cA$ be a chart defining this apartment, and $a \in \MS$ such that $f(a) = x$. We now define $\phi(x)$ to be the equivalence class of the pair $(f,\iota(a))$.

\begin{lemma}
The map $\phi$ is well-defined.
\end{lemma}
\proof
Let $x$ be a point of $X$, and $f$, $g$ two charts in $\cA$ such that for the two points $a$ and $b$ in $\MS$ it holds that $f(a)=g(b)=x$. Observe that $a \in Z:= f^{-1}(g(\MS))$ and that if $f\vert_Z =g \circ w \vert_Z$ for some $w \in W$, then $w.a =b$. This implies that $(f,\iota(a)) \sim (g,\iota(b))$ and that the map is well-defined.
\qed

\begin{lemma}
The map $\phi$ is an injection.
\end{lemma}
\proof
Let $x$ and $y$ be two points of $X$. By Condition (A3) there exists a chart $f\in \cA$ and two points $a,b$ of the model space $\MS$ such that $x=f(a)$ and $y=f(b)$. One has that $\phi(x)=\phi(y)$ if and only if $(f,a) \sim (f,b)$. Injectivity of $\phi$ now follows from Lemma~\ref{lemma:inj}.
\qed

We now check Condition (A2). We may assume without loss of generality that the two charts in $\cA'$ are of the form $\psi(f)$ and $\psi(g)$ with $f$ and $g$ two charts in $\cA$. Let $Z$ be the set $f^{-1}(g(\MS))$ and $Z'$ the set $\psi(f)^{-1}(\psi(g)(\MS'))$. A point $a$ lies in $Z'$ if there exists a point $b \in \MS'$ such that $(f,a) \sim (g,b)$. The necessary and sufficient condition for this equivalence to happen is that $a$ is an element of $\tilde{\iota}(Z)$. So we can conclude that $Z' =\tilde{\iota}(Z)$ and that $Z'$ is convex. The second part of Condition (A2) follows easily from Lemma~\ref{lemma:ltrans2}.

Using part two of Lemma~\ref{lemma:prepphi} one sees that an apartment of $(X,\cA)$ defines in a bijective way an apartment (i.e. an image of a chart in $\cA'$) of $(X',\cA')$. We can think of the map $\phi$ as embedding $X$ into a larger set $X'$ by adding points to each apartment. The above discussion for Condition (A2) implies that the intersection of two apartments before and after the embedding stays convex and also of the ``same shape'' (there is a map between both induced by $\tilde{\iota}$). So if two apartments share a Weyl chamber, then they also do after ``adding the extra points''.

We now prove (A6), as it is needed to prove certain results involving germs, which in turn are needed to prove (A3).

\begin{lemma}
Let $A'$, $B'$ and $C'$ be three apartments of $(X',\cA')$ intersecting pairwise in half-apartments. Then the intersection of all three is non empty.
\end{lemma}
\proof
Let $A$, $B$ and $C$ be the three apartments of $(X,\cA)$ which define respectively $A'$, $B'$ and $C'$. By the construction it follows that $A$, $B$ and $C$ also intersect pairwise in half-apartments, so they will have some point $x$ in common because of Condition (A6) for $(X,\cA)$. The point $\phi(x)$ lies in both $A'$, $B'$ and $C'$, so the intersection of these three apartments is not empty.
\qed

In order to prove the remaining conditions, we look at Weyl chambers in $(X',\cA')$ and investigate both their structure at infinity and the local structures. 

It is easily seen that we can define Weyl chambers in the apartments of $(X',\cA')$. Let $S$ be some Weyl chamber in some apartment $A'$ of $(X',\cA')$. Let $A$ be the corresponding apartment in $(X,\cA)$. One can consider parallel classes of Weyl chambers in this apartment $A$. Using the correspondence again, there corresponds a parallel class of Weyl chambers in $A'$ to each class in $A$. This way we can associate to the Weyl chamber $S$ in $A'$ a chamber $c$ in $\partial A \subset \partial_{\cA} X$, we say that: ``$S$ has the chamber $c$ at infinity''. From the already proven Condition (A2) it then follows that the associated chamber $c$ in $\partial_{\cA} X$ is independent of the choice of apartment containing $S$. 

\begin{lemma}\label{lemma:A4prep}
If two Weyl chambers $S_1$ and $S_2$ of $(X',\cA')$ have the same chamber at infinity, then they have a sub-Weyl chamber in common.
\end{lemma}
\proof
Suppose one has two Weyl chambers $S_1$ and $S_2$ (respectively in apartments $A_1'$ and $A_2'$), both having the chamber $c \in \partial_{\cA} X$ at infinity. Let $A_1$ and $A_2$ be the apartments of $(X,\cA)$ corresponding to respectively the apartments $A_1'$ and $A_2'$ of $(X',\cA')$. The apartments at infinity $\partial A_1$ and $\partial A_2$ have the chamber $c$ in common, so $A_1$ and $A_2$ both contain a Weyl chamber $T$ (with some base point $x \in A_1\cap A_2$) having $c$ at infinity. We go back to $(X',\cA')$ and obtain a Weyl chamber $T'$ based at $\phi(x)$ in both $A_1'$ and $A_2'$ having $c$ at infinity. This Weyl chamber $T'$ has with both Weyl chambers $S_1$ and $S_2$ sub-Weyl chambers in common, which on its turn implies that $S_1$ and $S_2$ have a sub-Weyl chamber in common.
\qed

We use the above lemma to show Condition (A4). Let $S_1$ and $S_2$ be two Weyl chambers in $(X',\cA')$, let $c_1$ and $c_2$ be the two corresponding chambers at infinity in $\partial_{\cA} X$. As $\partial_A X$ is a spherical building, there exists an apartment $A$ in $(X,\cA)$ such that $\partial A$ contains both chambers $c_1$ and $c_2$. Let $A'$ be the apartment in $(X',\cA')$ corresponding with $A$. This apartment $A'$ contains Weyl chambers $S_1'$ and $S_2'$ having respectively chambers $c_1$ and $c_2$ at infinity. Lemma~\ref{lemma:A4prep} proves that $A'$ contains sub-Weyl chambers of both $S_1$ and $S_2$, and so Condition (A4) holds.

One can define germs in $X'$ as well, and so also residues $\Delta_{x'} X'$ for points $x' \in X'$. We now list some lemmas, already listed in the introduction, which stay true in this case and are also proved in the same way.

\begin{lemma}\label{lemma:inap2}
Let  $c$ be a chamber of $\partial_\cA X$. For a Weyl chamber $S$ based at a point $x\in X'$ there exists an apartment $A$ with chart in $\cA'$ containing a germ of $S$ at $x$ and such that $c$ is contained in the boundary $\partial A$.
\end{lemma}
\proof
See Proposition~\ref{prop:tec}.
\qed

\begin{lemma}\label{lemma:is_spheric}
For all $x\in X'$ the residue $\Delta_xX'$ is a spherical building of type $\RS$.
\end{lemma}
\proof
See Proposition~\ref{prop:exis_res}.
\qed

\begin{lemma}\label{lemma:cotje}
Let $S,T$ be two Weyl chambers based at the same point $x \in X'$. If their germs are opposite in $\Delta_xX'$ then there exists a unique apartment of $(X',\cA')$ containing $S$ and $T$.
\end{lemma}
\proof
See Proposition~\ref{prop:list}, item~\ref{list:5.23}.
\qed

The above lemmas allow us to reason with residues, which is crucial in proving the following covering result. 

\begin{prop}\label{prop:cover}
Let $A_1$ and $A_2$ be two apartments of $(X',\cA')$. Let $\mathcal{C}\define \mathcal{C}(A_1, A_2)$ be the set of apartments containing at infinity two chambers $c_i\in\partial A_i$, $i=1,2$,  which are opposite in $\partial_\cA X$. Then $\mathcal{C}$ is a finite set of apartments such that each pair of points $(x,y)\in A_1\times A_2$ is contained in one of these apartments in $\mathcal{C}$.
\end{prop}
\proof
Consider two points $x_1 \in A_1$ and $x_2 \in A_2$. We first prove that if these points are contained in a common apartment $B$, then they also lie in an apartment in $\mathcal{C}$. We choose an $x_1$--based Weyl chamber $S_1$ in $B$ containing $x_2$. Lemma~\ref{lemma:is_spheric} implies the existence of an $x_1$--based Weyl chamber $T_1$ contained in $A_1$ whose germ is opposite the germ of $S_1$ at $x_1$. 
By Lemma~\ref{lemma:cotje} the set $T_1\cup S_1$ will be contained in some apartment $B'$. 

Let $T_2$ be the $x_2$--based Weyl chamber parallel to $T_1$, i.e. $\partial T_2=\partial T_1$. It is easy to see that $T_2$ contains $x_1$. We denote by $T_2'$ the $x_2$--based Weyl chamber in $A_2$ whose germ at $x_2$ is opposite $\Delta_{x_2}T_2$. 
Again by Lemma~\ref{lemma:cotje} we obtain a unique apartment $B''$ containing $T_2$ and $T'_2$.  This apartment $B''$ contains $x_1$ and $x_2$. Since apartments of $(X',\cA')$ are in one-to-one correspondence with apartments in $\partial_\cA X$ the apartment $B''$ is uniquely determined by the chambers $\partial T_2$ and $\partial T_2'$ in its boundary. By construction the apartment $B''$ is contained in $\mathcal{C}$. So we conclude that if two points $x_i\in A_i$, $i=1,2$ lie in a common apartment, then they also lie in a common apartment in $\mathcal{C}$.

Now suppose that the two points $x_i\in A_i$ do not lie in one apartment. Let $K$ be the set of points in $A_1$ which do lie in an apartment with $x_2$. Note that due to the above discussion and the fact that (A2) is already proven for $(X',\cA')$ we have that $K$ is a finite union of convex sets. Let us show that $K$ is non-empty. Let $\partial S$ be a chamber at infinity of the apartment $A_1$. Lemma~\ref{lemma:inap2} implies that there exists an apartment $A$ in $(X',\cA')$  containing $x_2$ and $\partial S$ at infinity. By applying Lemma~\ref{lemma:A4prep} one obtains that $A$ and $A_1$ in $(X',\cA')$ share at least a Weyl chamber. So $K$ is not empty, but also not the entirety of $A_1$ as $x_1 \notin K$. 

Because $K$ is a finite union of convex sets, one can find a point $y$ in $K$ such that not all the germs based at $y$ lie in $K$ (a point at the ``border'' of $K$). Let $\Delta_y R$ be such a germ, and $c$ a chamber at infinity of a Weyl chamber $U$ based at $y$ containing $x_2$ (possible because there exists an apartment containing both). Lemma~\ref{lemma:inap2} yields that there exists an apartment $A'$ containing the germ $\Delta_y R$ and $c$ at infinity. Apartment $A'$ also contains the Weyl chamber $U$ (so $x_2$ as well) due to (A2). The germ $\Delta_y R$ lies in $K$, contradicting the way we have chosen $\Delta_y R$. So we obtain that $K$ contains all points of $A_1$. The proposition is hereby proven.
\qed

As a consequence of this covering result we are now able to prove the two remaining conditions.

\begin{cor}
The object $(X',\cA')$ satisfies Conditions (A3) and (A5).
\end{cor}
\proof
The above proposition implies that a pair of points always is contained in a common apartment, which exactly is (A3). It also follows that given some germ $\mu$ of a Weyl chamber of $X'$ and point $x \in X'$ one can as well find an apartment containing these both. So one can define retraction-like maps as in Definition~\ref{Def_vertexRetraction}. As we have already proven Condition (A2) these maps are well-defined. They are indeed retractions reusing the same argument found in Appendix C of \cite{PetraThesis}. Hence (A5) is satisfied.
\qed

So we have proven that $(X',\cA')$ is indeed a $\Lambda'$--building. The next lemma states a link between the distance functions of both generalized affine buildings.

\begin{lemma}
Let $x$ and $y$ be two points of $X$, then $d'(\phi(x),\phi(y))= e(d(x,y))$, where $d$, $d'$ are the distance functions defined respectively on $(X,\cA)$ and $(X',\cA')$. 
\end{lemma}
\proof
Let $f\in \cA$, $a,b \in \MS$ such that $x=f(a)$ and $y=f(b)$ (possible by Condition (A3)). Then $d(x,y)$ equals $d(a,b)$, and $d'(\phi(x),\phi(y))$ equals $d'(\iota(a),\iota(b))$. The equality we need to prove follows from the discussion in the beginning of the section.
\qed

This concludes the proof of the first assertion of the third Main Result~\ref{theorem:thirdmain}. 

For two $\Lambda$--buildings $(X_1,\cA_1)$ and $(X_2,\cA_2)$ with an isometric embedding $\psi$ from $X_1$ to $X_2$ mapping apartments to apartments it is easy to find a map $\phi'$ between the resulting $\Lambda'$--buildings $(X_1',\cA_1')$ and $(X_2',\cA_2')$ (also mapping apartments to apartments) by the base change functors $\phi_1$ and $\phi_2$ such that the following diagram commutes:
$$
\xymatrix{
	X _1\ar[r]^{\psi}\ar[d]_{\phi_1} 	& X_2 \ar[d]^{\phi_2}\\
	X'_1 \ar[r]^{\psi'} & X'_2  
}
$$
This assures functoriality.


\section{Applications}\label{sec:applications}
In this section we prove some applications of the main results. A summary of these can be found in Section~\ref{section:pres_apps}.

\subsection{Asymptotic cones of $\Lambda$--buildings}\label{sec:cones}

We will use the base change functor to prove that asymptotic cones as well as ultracones of generalized affine buildings are again generalized affine buildings.

An \emph{ultrafilter} $\mu$ on an infinite set $I$ is a finite additive measure $\mu:2^I\to \{0,1\}$ such that $\mu(I)=1$. We say that $\mu$ is \emph{principal} if it is a Dirac measure, i.e. concentrated in one element of $I$. 
Sometimes we identify $\mu$ with the subset $F_\mu$ of the powerset $2^I$ consisting of the sets having measure one. Therefore another way of stating that $\mu(A)=1$ is to say that $A$ is contained in $F_\mu$. To simplify (and slightly abuse) notation we write $A\in \mu$ instead of $A\in F_\mu$. 

\begin{definition}
Given an ultrafilter $\mu$ on an infinite set $I$ we may define the \emph{ultraproduct} of a sequence of sets $(X_i)_{i\in I}$ as follows:
Let $\tilde X\define \Pi_{i\in I} X_i$, be the Cartesian product of the $X_i$. We define a relation $\sim$ on $\tilde X$ by
$$
x\sim y \;\Leftrightarrow\; \mu(\{i\in I \;:\; x_i=y_i\} )=1.
$$
It is easy to see that $\sim$ is an equivalence relation on $\tilde X$. The \emph{ultraproduct}  $\uprod X$ of the sequence  $(X_i)_{i\in I}$ is the space of equivalence classes $\uprod x$ of sequences $x=(x_i)_{i\in I}$ with respect to this equivalence relation.
In case that $X_i=X$ for all $i\in I$ we call $\uprod X$ the \emph{ultrapower} of $X$.
\end{definition}

{\L}o\v{s}'s Theorem, see 2.1 on page 90 in \cite{BellSlomson}, tells us that anything which may be stated in a first order language on the level of the components $X_i$ is still true for the ultraproduct of the sequence.

One easily observes, defining multiplication componentwise, that the ultraproduct of a sequence of abelian groups $\Lambda_i$ is again an abelian group. If each $\Lambda_i$ is totally ordered, then the ultraproduct carries a natural ordering defined as follows. An element $\uprod(a_n)_{n\in I}$ of $\uprod\Lambda$ is \emph{smaller or equal} to $\uprod(b_n)_{n\in\N}\in{\uprod\Lambda}$ if the set $\{n\in  I : a_n\leq b_n\}$ has measure one.

The ultraproduct $\uprod X$ of a sequence $(X_i, d_i)_{i\in I}$ of $\Lambda_i$--metric spaces carries a natural $\uprod\Lambda$--valued metric. The distance function $\ud$ on $\uprod X$ is defined by 
$$
\ud({\uprod} x, {\uprod} y)\define {\uprod} (d_i(x_i, y_i) )_{i\in I} \;\;\in{\uprod} \Lambda,
$$
which is simply the equivalence class of the sequence of distances of the components $x_i$ and $y_i$ in the corresponding metric space $X_i$.

Fix a spherical root system $\RS$ and denote by $\sW$ its spherical Weyl group. We consider a sequence $(X_i,\cA_i)_{i\in I}$ of generalized affine buildings, where $(X_i,\cA_i)$ is modeled on $\MS_i=\MS(\RS,\Lambda_i)$, with $\Lambda_i$ a totally ordered abelian group. Here it is important that the underlying spherical root system is the same for all factors.
Further assume that each $X_i$ carries a $\Lambda_i$--valued metric $d_i$. 

We will prove that the ultraproduct $\uprod X$ of $(X_i,\cA_i)_{i\in I}$ is again a generalized affine building. In order to do so, we need to take the product of all structural features, such as  charts or apartments, the Weyl groups, the distance function, \dots{ }simultaneously.
This process might be formalized in terms of first order languages and formulas on a certain set which involves all these structures. Taking the product of such a ``universal setting'' will allow us to talk about the same structures in the ultraproduct which we already had in the components themselves. However, we will not carry out the details here. We refer the interested reader to chapter 5 of Bell and Slomson's book \cite{BellSlomson} on models and ultraproducts.

\begin{prop}\label{prop:ultraproducts}
The ultraproduct $(\uprod X,{\uprod}\cA)$ is a generalized affine building  modeled on $\MS(\RS, {\uprod}\Lambda)$ and carries a natural $\uprod\Lambda$--metric $\uprod d$ induced by $d$. 
\end{prop}
\proof
The points and charts of $(\uprod X,{\uprod} \cA)$ are equivalence classes of sequences $(x_i)_{i\in I}$ (resp. $(f_i)_{i\in I}$), where $x_i$ is a point in $X_i$ and $f_i$ a chart in $\cA_i$.

First we will prove that the ultraproduct $\uprod \MS$ of the model spaces $\MS_i$ equals $\MS(\RS, {\uprod}\Lambda)$. Identifying points of $\MS_i$ with elements of the full affine Weyl group $\aW_i\define \sW\ltimes\Lambda_i^n$, with $n=\mathrm{rank}(\RS)$, this question can be answered by proving that $\uprod\aW\define {\uprod ( \aW_i)_{i\in I}} = \sW\ltimes ({\uprod\Lambda} )^n$.

Abbreviate $\Lambda_i^n$ by $T_i$.  Since the spherical Weyl group is finite, every sequence $(w_i)_{i\in I}$, with $w_i\in\sW$, is equivalent to a constant sequence $(u)_{n\in I}$. Hence given representatives $(u_i, s_i)_{i\in I}$ and $(v_i, t_i)_{i\in I}$ of elements $\uprod u, \uprod v$ of $\uprod \aW$, we observe
$$
 (u_i, s_i)_{i\in I}  \cdot (v_i, t_i)_{i\in I}
	=  (u_i v_i,\; s_i + r_{u_i}(t_i) )_{i\in I} 
	\sim ( uv, s_i + r_{u}(t_i))_{i\in I} 
$$
The equivalence class w.r.t. $\mu$ of the sequence $( uv, s_i + r_{u}(t_i))_{i\in I} $ is obviously an element of $ \sW\ltimes {\uprod T}.$ 

Interpreting Conditions (A1) to (A6) componentwise on sequences and using the fact that each factor $(X_i, \cA_i)$ satisfies these Conditions it is easy to deduce from  {\L}o\v{s}'s Theorem, that the ultraproduct satisfies all the Conditions as well. Hence the claim.
\qed

From now on we will restrict ourselves to ultrapowers. Fix once and for all an ordered abelian group $\Lambda$. 
The subset of \emph{finite elements} $\uprod\Lambda_{fin}$ of the ultrapower of an ordered abelian group $\Lambda$ is defined by
$$
\uprod\Lambda_{fin}=\{\alpha\in{\uprod\Lambda} : -c< \alpha <c \;\text{ for some } c\in\Lambda_+\}.
$$
An easy computation shows that $\uprod\Lambda_{fin}$ is a convex subgroup of $\uprod\Lambda$. 

\begin{definition}
We define an equivalence relation $\sim_{fin}$ on $\uprod X$ by 
$$
\uprod x\sim_{fin} {\uprod y} \;\Leftrightarrow\; \ud(\uprod x,\uprod y) \leq n \text{ for some } n\in\Lambda.
$$
The \emph{ultracone}  $\Ucone(X) = {\uprod X}/_{\sim_{fin}}$ of the metric space $X$ is again a metric space whose metric $\ud_{fin}$ is the $\uprod\Lambda/{\uprod\Lambda_{fin}}$--quotient-metric induced by $\ud$. 
\end{definition}

\begin{definition}
Let $\mu$ be a non-principal ultrafilter, $\alpha=(\alpha_i)_{i\in I}$ a non-finite element of $\uprod \Lambda$ and assume that $(X,d)$ is a $\Lambda$--metric space with chosen base point $o$. Clearly $\uprod X$ is a $\uprod \Lambda$--metric space. Set 
$
X^{<\alpha>}\define\{ \uprod x\in {\uprod X} \,:\, \ud (\uprod x,o) \leq n \alpha \,\text{ for some } n\in\Lambda \}.
$
We define $\uprod x \sim_\alpha {\uprod y} $ on $X^{<\alpha>}$ if $\ud(\uprod x, {\uprod} y) \,<\, \alpha\cdot\alpha_i^{-1}$ for all $i\in I$.
The set 
$$
\Cone(X)\define X^{<\alpha>}/_{\sim_\alpha}
$$
is called the \emph{asymptotic cone} of $X$. It carries a natural $\Lambda$--metric defined by 
$$
d(x,y)=\mathsf{std}( \ud(x,y)\cdot\alpha^{-1} ),
$$
where $\mathsf{std}$ denotes the standard part of the element $ \ud(x,y)\cdot\alpha^{-1}$. 
\end{definition}

An element $\lambda\in\Lambda$ is the \emph{standard part} of an element $\uprod\lambda$ in $\uprod\Lambda_{fin}$ if the absolute value of the difference satisfies  $\vert  \uprod\lambda - \lambda \vert < \alpha\cdot \alpha_i^{-1}$ for all $i\in I$. 

\begin{theorem}\label{thm:cones}
The class of generalized affine buildings is closed under taking asymptotic cones and ultracones. 
Furthermore, if $(X,\cA)$ is modeled on $\MS(\RS,\Lambda)$, then $\Cone(X)$ has model space $\MS(\RS,\Cone(\Lambda))$ and $\Ucone(X)$ is modeled on $\MS(\RS,\Ucone(\Lambda))$.
\end{theorem}
\proof
Let $(X,\cA)$ be a generalized affine building modeled on $\MS(\RS,\Lambda)$ and let $o$ be a base point in $X$ and $d: X\times X\to\Lambda$ a metric.

First simply view $X$ as a metric space and consider its ultracone $\Ucone(X)$. There is then an obvious projection $\pi$ from $\uprod X$ onto $\Ucone(X)$  such that 
$$
\ud_{fin}( \pi(\uprod x) , \pi(\uprod y) ) = e( \ud ( \uprod x, \uprod y) ),
$$
where $e$ is the projection from $\uprod \Lambda$ to $\uprod\Lambda_{fin}$.
By 1 of our Main Result~\ref{theorem:firstmain} there exists a $\uprod\Lambda_{fin}$--affine building $(X',\cA')$, which is the image of the base change functor $\phi$ associated to $e:\uprod\Lambda\to \uprod\Lambda_{fin} $. Using the observations made above we may deduce from 2 of our first Main Result that $\Ucone(X)$ is isomorphic to $X'$ and carries the structure of a $\uprod\Lambda_{fin}$--affine building.

We now apply similar arguments in the case of asymptotic cones. By our second Main Result \ref{theorem:secondmain} the set $X''\define\phi^{-1}(\phi_e(0))$ is a $\ker(e)$--affine building. It is easy to see that $X''$ coincides with $X^{<\alpha>}$, which was defined to be the set of sequences whose distance to the constant sequence $o$ is an element of $\uprod\Lambda_{fin}$, that is a finite element of $\uprod\Lambda$.

The asymptotic cone of $X$ is the quotient space of $X^{<\alpha>}$ by the relation $\sim_\alpha$, that is identifying sequences having infinitesimal distance to $o$. This corresponds to the image of the base change functor $\phi'$ associated to the projection $e':\uprod\Lambda_{fin} \to \Cone(\Lambda)$. Hence by Main Result~\ref{theorem:firstmain} the asymptotic cone $\Cone(X)$ is isomorphic to the $\Cone(\Lambda)$--affine building $\phi'(X^{<\alpha>})$.
\qed

\subsection{Reducing to the $\R$-building case: a fixed point theorem}\label{section:reducing}
This last application is an example of how our base change functors can reduce problems of $\Lambda$-buildings to the more known and familiar case of $\Lambda = \R$. We will demonstrate this by proving a fixed point theorem for certain $\Lambda$-buildings. The main tool herein is an embedding theorem by Hahn. First we need to define the Hahn product. Given an ordered set $I$ and collection $(\Lambda_i)_{i \in I}$ of ordered abelian groups of order, then the \emph{Hahn product} is the subgroup of $\prod_{i \in I} \Lambda_i$ where the set $I' \subset I$ of indices with non-zero entries of element is always well-ordered w.r.t. the reverse ordering of $I$ (this means that each non-empty subset of $I'$ has a maximal element). This subgroup carries a natural lexicographical ordering. 

\begin{theorem}[Hahn's embedding theorem,~\cite{Hahn}]
Given an ordered abelian group $\Lambda$, then there exists an ordered set $I$ such that $\Lambda$ is isomorphic as ordered abelian group to a subgroup of the Hahn product of copies of the real numbers $\R$ over an ordered set $I$.  
\end{theorem}

Let $(X,\cA)$ be a $\Lambda$--building. By the third Main Result and the above theorem one can embed this building in a $\Lambda'$--building where $\Lambda'$ is the Hahn product of copies of the real numbers $\R$ over an ordered set $I$. So assume $\Lambda$ is of this form. Additionaly assume that $I$ is well-ordered (so every non-empty subset of $I$ has a least element). Our fixed point theorem is now as follows.

\begin{theorem}\label{theorem:fixie}
A finite group of isometries of a generalized affine building $(X,\cA)$ with the above properties admits a fixed point.
\end{theorem}
\proof

For an arbitrary non-zero element $h$ of $\Lambda$ let $i_h \in I$ be the maximal element of the set of indices in $I$ with non-zero entries in the representation of $h$ as a product. If $h$ is zero, then we set $i_h$ to be $-\infty$. For a non-zero element $g$ of $\Lambda$, one can now define the following two convex subgroups of $\Lambda$: 
\begin{align*}
M_g := \{\lambda \in \Lambda: i_\lambda \leq i_g\}, \\
N_g := \{\lambda \in \Lambda: i_\lambda < i_g\}. 
\end{align*}
Note that $M_g \diagup N_g$ is isomorphic to $\R$. 

Choose some point $x_0$ of $X_0:=X$, the orbit of $x_0$ is finite, hence it is bounded by some element $g_0 \in\Lambda$. If $g_0$ is zero then we have found a fixed point, so suppose this is not the case. The points of $X$ at distance in $M_{g_0}$ from $x_0$ form a $M_{g_0}$--building (by applying the second Main Result on the canonical epimorphism $\Lambda \rightarrow \Lambda \diagup M_{g_0}$). Note that $G$ stabilizes this $M_g$--building. Consider the canonical epimorphism $M_{g_0} \rightarrow M_{g_0} \diagup N_{g_0}$, so using the first Main Result we obtain an $\R$--building on which $G$ acts, still as a finite group of isometries (we mark this step by (*) for further reference). By a result of the second author, this action has a fixed point $y_0$ (see~\cite{KS}). In the original $\Lambda$--building $(X,\cA)$ this point $y_0$ corresponds to a set $X_1$ of points with distance in $N_{g_0}$ from a certain point $x_1$. By the second Main Result we can consider the set $X_1$ as an $N_{g_0}$--building stabilized by $G$. So we can repeat this algorithm. Because $I$ is well-ordered, the algorithm has to stop at some point where $g_i = 0$. When this happens we have obtained a fixed point of $G$ in $(X,\cA)$.
\qed

This result should not be interpreted as a full investigation into fixed point theorems of $\Lambda$--buildings, but as a quick example of how our results combined with Hahn's embedding theorem can reduce problems to the $\R$--building case. Even when the theorems one obtains in this way do not hold in full generality (like this fixed point theorem), they might help to better understand generalized affine buildings and point out where possible difficulties might occur.

We end with an example of this last thing. In~\cite{KS} there is proved that a finitely generated bounded group of isometries of an $\R$-building has a fixed point, so what happens if we only ask this weaker condition to be fulfilled in the above theorem? An analoguous proof would fail in step (*), one cannot show that this new action is bounded. Indeed, consider the lexicographically ordered abelian group $G := \R \times \R$ as a generalized affine building with only one apartment, and define for each $k \in \mathbb{Z}$ the isometry 
$G\rightarrow G: (x,y) \mapsto (x,y+kx)$. All these isometries form a finitely generated bounded group (as each orbit is bounded by $(1,0)$), but it has no fixed point. So this exercise shows that there lies a difficulty in the notion of boundedness.


\renewcommand{\refname}{Bibliography}
\addcontentsline{toc}{section}{\refname}
\bibliography{bibliography}
\bibliographystyle{alpha}

\end{document}